\documentclass[letterpaper, 12pt]{article}

\usepackage{natbib, subfigure, color}
\usepackage{amsthm,amsfonts,amsmath,enumerate,amscd, multirow, graphicx, blindtext}

\usepackage[letterpaper]{geometry}
\newtheorem{corollary}{Corollary}[section]

\newtheorem{proposition}{Proposition}[section]
\begin{document}
\begin{titlepage}

\begin{center}
\textbf{\Large{Cross-Training with Imperfect Training Schemes}}
\end{center}
\begin{center}
Burak B\"uke \\
School of Mathematics, The University of Edinburgh,\\ Edinburgh, UK EH9 3JZ\\Phone:+44-131-650-5086\\ Fax:+44-131-650-6553\\ E-mail: B.Buke@ed.ac.uk
\vspace{0.2in}

\"{O}zg\"{u}r  M.\ Araz\\
College of Public Health,\\
Department of Health Promotion, Social and Behavioral Health,\\
University of Nebraska Medical Center,\\
 Omaha, NE 68198, \\
Phone:+1-402-559-5907 \\ Fax:+1-402-559-3773\\ E-mail: 
 ozgur.araz@unmc.edu
 \vspace{0.2in}
 
 John W. Fowler\\
W. P. Carey School of Business,
 Arizona State University,\\
 Tempe, AZ 85287\\
Phone:+1-480-965-3727 \\ Fax:+1-480-965-8629\\ E-mail: 
john.fowler@asu.edu
\end{center}
\end{titlepage}
\newpage

\begin{center}
\textbf{\Large{Cross-Training with Imperfect Training Schemes}}
\end{center}

\begin{abstract}
Cross-training workers is one of the most efficient ways to achieve flexibility in manufacturing and service systems to increase responsiveness to demand variability. However, it is generally the case that cross-trained employees are not as productive as employees who are originally trained on a specific task. Also, the productivity of the cross-trained workers depend on when they are cross-trained. In this work, we consider a two-stage model to analyze the affect of variations in productivity levels of  workers on cross-training policies. Our results indicate that the most important factor determining the problem structure is the consistency in productivity levels of workers trained at different times. As long as cross-training can be done in a consistent manner, the productivity differences between cross-trained workers and workers originally trained on the task plays a minor role. We also analyze the effect of the variabilities in demand and producivity levels. We show that if the productivity levels of workers trained at different times are consistent, the decision maker is inclined to defer the cross-training decisions as the variability of demand or productivity levels increases. However, when the productivities of workers trained at different times differ, the decision maker may prefer to invest more in cross-training earlier as variability increases.
\end{abstract}

\vspace{-0.2in}

\section{Introduction}
\vspace{-0.1in}

Designing flexible systems is one of the key strategies to increase responsiveness to variability in the market without sacrificing efficiency of the system. One way to achieve flexibility in a system is to cross-train workers on several processes. Cross-training has been proven to be highly beneficial in many different business environments, including but not limited to the semiconductor and automative industries, call centers and healthcare. For example, in the semiconductor industry, machine operators are often cross-trained to run more than one type of sophisticated equipment and technicians are often cross-trained to maintain more than one type of machine. In addition to increasing efficiency, cross-training can help  keep the budgets low, increase a company's ability to pay more to the employees, to reduce turnover rate, and increase quality due to the workers' ability to react to unexpected changes (see e.g., \cite{lyons92}, \cite{mccune94}, \cite{irakolvan07}). 

Cross-trained workers can be shifted to work on new tasks when needed, which yields a more efficient usage of the resources. However, it is generally the case that cross-trained employees do not perform equally well as employees who are originally trained on a specific task, i.e., the training schemes may be imperfect. Moreover, the productivity levels of the cross-trained workers may depend on when the cross-training is done. In this paper, our main goal is to analyze how these differences in productivity levels affect cross-training decisions. To analyze the effect of imperfect schemes and timing of cross-training decisions, we consider a two-stage model and study the problem in a newsvendor network setting, introduced by~\cite{van98} and~\cite{vanrud02}. Our model is similar to those of~\cite{van98} and~\cite{basranvan09b}. Similar to the prior work, the decision maker decides on the number of employees to cross-train, i.e., the level of flexibility, before realizing the demand for the products. In the prior work, the first stage decisions are structural (design) decisions where the level of flexibility is fixed and does not change in the future. However, after demand reveals the decision maker may see that additional cross-training is beneficial and may wish to increase the level of flexibility. Hence, we extend the two stage model of~\cite{van98} to allow the decision maker to do online cross-training in the second stage. As argued above, the productivities of the employees who are cross-trained before and after the demand is observed may differ. If the demand exceeds the capacity even after the second stage cross-training, the excess demand is lost and an opportunity cost is incurred.

Our main conclusion is that the main factor which determines the structure of the cross-training problem is whether the workers can be cross-trained with the same productivity consistently at different times. As long as the outcomes of the training schemes are consistent, the imperfections play a minor role and the cross-training decisions with imperfect training exhibit similar properties to the case where training schemes are perfect, and the results for perfect training schemes hold with minor modifications. We show that if the online training schemes are as effective as the offline schemes, the cross-training policies are independent of the opportunity cost. Similar to the abovementioned work, we provide newsvendor-type equations to characterize the optimal cross-training levels under different scenarios and use these equations to characterize situations when it is not profitable to invest in cross-training. Then, we investigate how the variability of demand and productivity factors affect cross-training policies. Under some mild conditions, we prove that the decision maker is inclined to decrease first stage cross-training when the variance of random parameters increases and the  first and second stage training schemes are consistent.  

If the first and second stage training schemes are not consistent in productivity, the opportunity cost plays an important role and the structure of the problem changes significantly, i.e., the results for perfect training schemes are no longer applicable. Our findings indicate that this structural difference is mainly due to the loss/gain of ``achievable capacity'' as a consequence of training workers at different times. To investigate this further, we assume that first stage cross-training is more effective on expected and test how the demand variability affects cross-training decisions. We see that for small demand variances, the achievable capacity may be enough to satisfy the demand with high probability even when all the training is done in the second stage. Hence, the opportunity cost plays a minor role and the first-stage cross-training decreases as demand variance increases. However, when demand variance is over a certain threshold, the opportunity cost becomes an important issue and should be avoided by increasing the achievable capacity. Hence, it may be beneficial to increase first-stage cross-training levels as variance increases beyond this threshold. We also study the behavior of cross-training policies as the variances of the first-stage productivity factors change. We observe that there is a critical threshold for the variance such that the first-stage productivity factor is less than the second stage productivity factor with a significant probability. Increasing variance of the first-stage productivity factor below this critical threshold leads to more first-stage cross-training. On the otherhand, increasing the variance above this critical threshold leads to a decrease in the optimal first-stage cross-training. 

The opportunity to benefit from flexibility without too much investment has recently accelerated research in designing efficient flexible systems and we see this work as a part of this research effort. In their seminal paper, ~\cite{jorgra95} show that almost all the benefits of a fully flexible system, where all resources can perform all tasks, can be achieved by using a moderate level of flexibility. Their results demonstrate that using a special flexibility configuration referred to as ``chaining'' and under certain assumptions on demand,  it is possible to obtain 98\% of the throughput of a fully flexible system using resources that can perform only two different tasks.  Using tools from queueing theory, \cite{jorinblu04} observe that cross-training can adversely affect performance if a poor control policy is used and demonstrate that a complete chain is robust with respect to the control policy and parameter uncertainty. In the literature, the ability of companies to achieve flexibility and efficiency while at the same time meeting customer needs is sometimes also refered as production agility~(\cite{gelhopvan07, hoptekvan04, hoppvanoyen04}). \cite{hoppvanoyen04} develop a framework for workforce cross-training, provide a comprehensive review of the recent literature and suggest some future research directions. \cite{hoptekvan04} and \cite{gelhopvan07} analyze flexibility decisions for  manufacturing systems operating under CONWIP or WIP-constrained policies and conclude that a cross-trained worker should perform her original task before helping on other tasks. \cite{pinshu00} perform numerical studies to analyze the trade-off between efficiency and quality due to cross-training. \cite{netdobshu02} show how the cross-training policies are affected by the demand correlation. \cite{DavisKher09}  indicate that under high workload imbalances, an extensive level of cross training is required to significantly improve the overall production performance. In a service environment, \cite{Gnanlet09} use two-stage stochastic programming model to determine optimal resource levels and demonstrate the benefits of cross-training activities in a health care setting. In their recent papers, ~\cite{basranvan09a, basranvan09b} define level-$k$ resources to be the resources that are able to process $k$ different tasks. In~\cite{basranvan09a}, they prove that for symmetric queueing systems one only need to use dedicated resources and level-2 resources. Similar to our analysis, \cite{basranvan09b} use the newsvendor network framework to prove that in the optimal flexibility configurations only two adjacent levels of flexibility are needed.~\cite{cct10} discuss the effect of production efficiency comparing full flexibility with chaining structure. Another paper which is closely related to our work is \cite{chaagni}, where they study the optimum fraction of flexible servers for a two task problem with perfect training schemes. They point to the fact that cross-trained workers may not be as efficient as dedicated workers. However, they do not provide an analysis of the problem.

There have been several attempts to formulate the problem as a mathematical program. \cite{brujoh98} and \cite{campbell99} provide integer programming models for cross-training workers with different capability levels. 
 \cite{wcfw00} develop a two-stage stochastic program for cross-training problem in semiconductor industry. \cite{Tanrisever2012334} propose a multi-stage stochastic integer programming model to design flexible systems in make-to-order environments. In the literature, methods including Markov decision processes, mathematical  programming and  heuristics, are used to schedule cross-trained workers (see e.g., \cite{vaiwin99}, \cite{saykar07} and \cite{Campbell11}). Our primary focus is on aggregate planning of cross-training efforts and we do not address the problem of scheduling the workers.
 \vspace{-0.2in}

\section{A Two-Stage Model for Cross-Training}
\label{Sc:TwoTasks}
\vspace{-0.1in}
In this section, our goal is to provide a detailed analysis of the cross-training problem for two tasks when the cross-training schemes are imperfect. We provide a two-stage model to answer how our cross-training policies are affected by the following factors:
\begin{enumerate}[i.]
\item the imperfectness in training schemes
\item the differences in productivity levels of workers cross-trained at different times
\item the variability in demand and productivity levels.
\end{enumerate} 
We also provide detailed numerical experiments in Section~\ref{Sec:NumRes}. 

We analyze the problem of cross-training workers between two tasks, $\alpha$ and $\gamma$. We assume that the capacity is measured in time units and initially there is $x^0_{\alpha}$ and $x^0_{\gamma}$ units of capacity dedicated to process tasks $\alpha$ and $\gamma$, respectively. The decision maker has to develop an aggregate workforce plan by cross-training some of the available workers before observing the capacity requirements $\tilde{d}_{\alpha}$ and $\tilde{d}_\gamma$. Before actual  demand is realized, it costs $c^1_{\alpha}$ to train one unit of dedicated capacity of $\gamma$ to work on task $\alpha$. The cross-trained workers can still work on their original task $\gamma$ with full capacity. However, they are not as efficient in their new skill $\alpha$, and their capacity needs to be adjusted by a productivity factor $\tilde{\delta}^1_{\alpha}$, where $0<\tilde{\delta}^1_{\alpha}\leq 1$; i.e., if a cross-trained $\gamma$-worker spends one hour working on task $\alpha$, it is equivalant to $\tilde{\delta}^1_{\alpha}$ hours of an original $\alpha$-worker. The productivity factor, $\tilde{\delta}^1_{\alpha}$, can also be perceived as an indicator of the effectiveness of the training program and is assumed to be random. After the capacity requirements for tasks are revealed, we first use the dedicated workers and workers cross-trained in the first stage to satisfy the demand. If the available capacity for task $\alpha$ is not enough to satisfy the requirement and there is excess capacity for task $\gamma$, additional cross-training can be performed online at a unit cost of $c^2_{\alpha}$. The productivity factor for the workers cross-trained in the second stage is $\tilde{\delta}^2_{\alpha}$, where $0<\tilde{\delta}^2_{\alpha}\leq 1$. If the workforce at hand cannot satisfy the capacity requirements even after the second stage cross-training, the demand is lost incurring  a unit opportunity cost of $h_{\alpha}$. A similar mechanism works to satisfy the demand for task $\gamma$ interchanging the subscripts. 

Without loss of generality, we assume that $h_{\alpha}\leq h_{\gamma}$.  To simplify our analysis and notation, we also assume that random variables $\tilde{d}_{\alpha}, \tilde{d}_\gamma, \tilde{\delta}^1=(\tilde{\delta}^1_{\alpha},\tilde{\delta}^1_{\gamma})$ and $\tilde{\delta}^2=(\tilde{\delta}^2_{\alpha},\tilde{\delta}^2_{\gamma})$ are continuous random variables with joint density function $f(\tilde{d}_{\alpha}, \tilde{d}_\gamma, \tilde{\delta}^1, \tilde{\delta}^2)$ and use $\tilde{\xi}=(\tilde{d}_{\alpha}, \tilde{d}_\gamma, \tilde{\delta}^1, \tilde{\delta}^2)$, whenever we do not need to address specific random variables. Throughout this work, a random variable $x$ is denoted $\tilde{x}$, $\mathbb{E}[\tilde{x}]$ denotes the expected value of the random variable. Similarly, we use $\mathbb{E}[\tilde{x};\Omega]=\int_{\Omega}xdP(x)$  to denote the expectation over a scenario region $\Omega$.

Now, we need to write out the objective function explicitly based on the mechanism described above. The decision maker initially decides on $x^1_{\alpha}$ and $x^1_{\gamma}$, which are the amount of workforce cross-trained from dedicated capacity of $\gamma$ and $\alpha$, respectively. We can decompose the cost function $g(x^1_{\alpha},x^1_{\gamma}, \tilde{\xi})$ into the first stage cost which is incurred due to initial cross-training and the second stage cost $v(x^1_{\alpha},x^1_{\gamma}, \tilde{\xi})$ which is revealed after the realization of random parameters. Hence,
\begin{equation}
\min_{x^1_{\alpha}, x^1_{\gamma}}\mathbb{E}[g(x^1_{\alpha},x^1_{\gamma}, \tilde{\xi})]=\min_{x^1_{\alpha}, x^1_{\gamma}}c^1_\alpha x^1_\alpha+c^1_\gamma x^1_\gamma+\mathbb{E}[v(x^1_{\alpha},x^1_{\gamma}, \tilde{\xi})].
\label{Eq:cost_function}
\end{equation}
The second stage cost, $v(x^1_{\alpha},x^1_{\gamma}, \tilde{\xi})$, depends on whether new cross-training is needed after the demand is observed. Hence, $v(x^1_{\alpha},x^1_{\gamma}, \tilde{\xi})$ takes different forms depending on the realization of parameters. To analyze this function further and calculate the expected value for given $x^1_{\alpha}, x^1_{\gamma}$, we first use the tower property
$\mathbb{E}[v(x^1_{\alpha},x^1_{\gamma}, \tilde{\xi})]=\mathbb{E}[\mathbb{E}[v(x^1_{\alpha},x^1_{\gamma},\tilde{\xi})|\tilde{\delta}^2,\tilde{\delta}^1]]$. To calculate the conditional expectation $\mathbb{E}[v(x^1_{\alpha},x^1_{\gamma},\tilde{\xi})|\tilde{\delta}^2, \tilde{\delta}^1]$, we first partition the support of $(\tilde{\delta}^1, \tilde{\delta}^2)$ into  subsets, where in each subset the nature of the second stage decision is different. Then, we partition the support of $\tilde{d}_{\alpha}$ and $\tilde{d}_\gamma$ so that the function $v(x^1_{\alpha},x^1_{\gamma}, \tilde{\xi})$ has a single form within a partition. This partitioning scheme is explained below and the most general graphical representation is given in Figure~\ref{Fg:Partition}.

We start partitioning the support of  productivity factors $(\tilde{\delta}^1, \tilde{\delta}^2)$ by considering whether $h_\alpha\geq \tilde{\delta}^1_\gamma h_{\gamma}$ or $h_\alpha<\tilde{\delta}^1_\gamma h_{\gamma}$. This criterion determines the task to allocate a  worker cross-trained to work on task $\gamma$, when we need the worker for both tasks.   If we allocate the worker for task $\alpha$, our gain will be $h_{\alpha}$. If instead, we allocate the worker to task $\gamma$, our gain will be $\tilde{\delta}^2_\gamma h_{\gamma}$. In the first case, where $h_\alpha\geq \tilde{\delta}^2_\gamma h_{\gamma}$, it is always more profitable to use the fully productive capacity of task $\alpha$ to satisfy $\tilde{d}_\alpha$ when needed. In the second case, where $h_\alpha<\tilde{\delta}^2_\gamma h_{\gamma}$, even if $x_{\alpha}^0$ is enough to satisfy the demand $\tilde{d}_{\alpha}$, it may be more profitable to shift some of this capacity to task $\gamma$ and lose task~$\alpha$ demand.

\begin{figure}
\begin{center}
\includegraphics[scale=0.5]{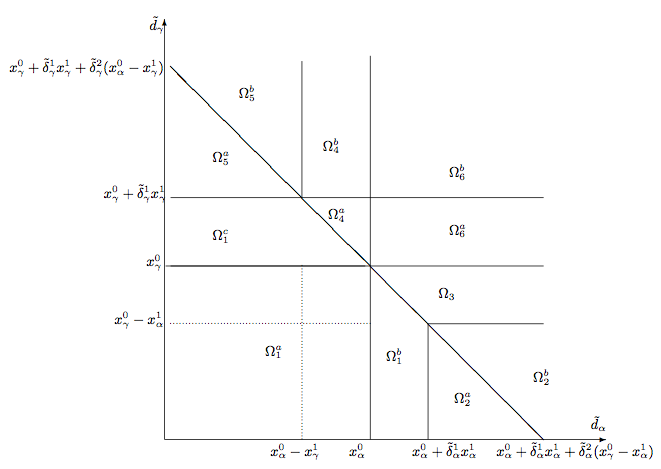}
\vspace{-0.2in}
\caption{Partitioning the support of the demand vector}
\vspace{-0.2in}
\label{Fg:Partition}
\end{center}
\end{figure}

\vspace{-0.2in}
\subsection{Case 1: Workers Used in Their Original Tasks $(h_\alpha\geq \tilde{\delta}^1_\gamma h_{\gamma})$}
\vspace{-0.1in}

Under the condition $h_\alpha\geq \tilde{\delta}^1_\gamma h_{\gamma}$, task $\alpha$ demand will be lost only when the demand $\tilde{d}_{\alpha}$ cannot be supplied by using the initial capacity $x_\alpha^0$ and cross-trained workers who are not allocated to task $\gamma$. Since $h_\alpha\leq h_{\gamma}$, the same claim is always true for task $\gamma$. Now we ask whether we shall resort to second stage cross-training. The answer depends on the value of second stage productivity factors and will be the next criterion for partitioning the support of $\tilde{\delta}^2$.

If we choose to cross-train workers in the second stage, we need to spend $c_{\alpha}^2/\tilde{\delta}_{\alpha}^2$ to satisfy unit demand of task $\alpha$, or we lose the demand and incur a cost of $h_{\alpha}$. Hence, for task $\alpha$, we choose to resort to second stage cross-training first if $c_{\alpha}^2\leq\tilde{\delta}_{\alpha}^2h_{\alpha}$ and we never cross-train in the second stage if $c_{\alpha}^2>\tilde{\delta}_{\alpha}^2h_{\alpha}$. Similar logic applies to task $\gamma$.
\vspace{-0.2in}

\subsubsection{Case 1.a: $c_{\alpha}^2>\tilde{\delta}_{\alpha}^2h_{\alpha}$ and $ c_{\gamma}^2>\tilde{\delta}_{\gamma}^2h_{\gamma}$.}
\vspace{-0.1in}

For this case, losing the demand is preferable over second-stage cross-training for both tasks. Hence, there will be no second stage cross-training, and we can use the following partitioning to explicitly state $v(x^1_{\alpha},x^1_{\gamma},\tilde{\xi})$ for any given $x_\alpha^0, x^0_{\gamma}$ and $\tilde{\xi}$.

\begin{enumerate}
\item$\Omega_1^a=\{(\tilde{d}_{\alpha},\tilde{d}_\gamma): \tilde{d}_{\alpha}\leq x^0_\alpha, \tilde{d}_{\gamma}\leq x^0_\gamma\}$. For the scenarios in $\Omega_1$, the initial workforce is enough to satisfy the capacity requirements. Hence, $v(x^1_{\alpha},x^1_{\gamma}, \tilde{\xi})=0$ on $\Omega_1^a$.

\item$\Omega_1^b=\{(\tilde{d}_{\alpha},\tilde{d}_\gamma): x^0_\alpha<\tilde{d}_{\alpha}\leq \min\{x^0_\alpha+\tilde{\delta}^1_\alpha x^1_{\alpha}, x^0_\alpha+\tilde{\delta}^1_\alpha (x^0_{\gamma}-\tilde{d}_{\gamma})\},\tilde{d}_{\gamma}\leq x^0_\gamma\}$. On $\Omega_1^b$, the initial workforce $x^0_\alpha$ cannot satisfy $\tilde{d}_{\alpha}$. Task $\gamma$ may need to use some workers who are cross-trained to work on $\alpha$, but the remaining cross-trained workforce is enough to satisfy the excess demand for $\alpha$. Hence, again $v(x^1_{\alpha},x^1_{\gamma}, \tilde{\xi})=0$ on $\Omega_1^b$.

\item $\Omega_1^c$ is defined similar to $\Omega_1^b$ with $\alpha$ and $\gamma$ interchanged, and $v(x^1_{\alpha},x^1_{\gamma}, \tilde{\xi})=0$ on $\Omega_1^c$.

We define $\Omega_1=\Omega_1^a\cup\Omega_1^b\cup\Omega_1^c$. When the demand falls in this region, no recourse action is needed in the second stage and hence no cost is incurred.

\item $\Omega_2=\{(\tilde{d}_{\alpha},\tilde{d}_\gamma):x^0_\alpha+\tilde{\delta}^1_\alpha x^1_{\alpha}<\tilde{d}_{\alpha}, 0\leq \tilde{d}_{\gamma}\leq x_{\gamma}^0-x_{\alpha}^1\}(=\Omega_2^a\cup\Omega_2^b$ in Figure~\ref{Fg:Partition}). The demand for $\gamma$ is low so that all the cross-trained workers can be used to work on task $\alpha$. However, even this is not enough to satisfy $\tilde{d}_{\alpha}$ and since second stage cross-training is not profitable, demand is lost. Hence, $v(x^1_{\alpha},x^1_{\gamma}, \tilde{\xi})=h_{\alpha}(\tilde{d}_{\alpha}-x_{\alpha}^0-\tilde{\delta}_\alpha^1 x_{\alpha}^1)$ on $\Omega_2$.

\item $\Omega_3=\{(\tilde{d}_{\alpha},\tilde{d}_\gamma), x^0_\alpha+\tilde{\delta}^1_\alpha (x^0_{\gamma}-\tilde{d}_{\gamma})<\tilde{d}_{\alpha}, x^0_\gamma-x^1_{\alpha}<\tilde{d}_{\gamma}\leq x^0_\gamma\}$. For scenarios in this subset of the support, some but not all of the workers who are cross-trained to work on task $\alpha$ can be shifted to $\alpha$, and this is not enough to satisfy the capacity requirement. The excess demand is lost and hence,  
$v(x^1_{\alpha},x^1_{\gamma}, \tilde{\xi})=h_\alpha(\tilde{d}_{\alpha}-(x^0_\alpha+\tilde{\delta}^1_\alpha (x^0_{\gamma}-\tilde{d}_{\gamma}))) \mbox{ on $\Omega_3$}.$

\item $\Omega_4=\{(\tilde{d}_{\alpha},\tilde{d}_\gamma), x^0_\alpha-x^1_{\gamma}<\tilde{d}_{\alpha}\leq x^0_\alpha, x^0_\gamma+\tilde{\delta}^1_\gamma (x^0_{\alpha}-\tilde{d}_{\alpha})<\tilde{d}_{\gamma}\}(=\Omega_4^a\cup\Omega_4^b$ in Figure~\ref{Fg:Partition}). This region is defined in the same way as $\Omega_3$ where $\alpha$ and $\gamma$ are interchanged. The second stage cost function is given by
$v(x^1_{\alpha},x^1_{\gamma}, \tilde{\xi})=h_\gamma(\tilde{d}_{\gamma}-(x^0_\gamma+\tilde{\delta}^1_\gamma (x^0_{\alpha}-\tilde{d}_{\alpha}))).$

\item $\Omega_5=\{(\tilde{d}_{\alpha},\tilde{d}_\gamma): 0\leq \tilde{d}_{\alpha}\leq x_{\alpha}^0-x_{\gamma}^1, x^0_\gamma+\tilde{\delta}^1_\gamma x^1_{\gamma}\}(=\Omega_5^a\cup\Omega_5^b$ in Figure~\ref{Fg:Partition}). $\Omega_5$ is the analog of $\Omega_2$ with $\alpha$ and $\gamma$ interchanged. Hence, $v(x^1_{\alpha},x^1_{\gamma}, \tilde{\xi})=h_{\gamma}(\tilde{d}_{\gamma}-x_{\gamma}^0-\tilde{\delta}_\gamma^1 x_{\gamma}^1)$ on $\Omega_5$.

\item $\Omega_6=\{(\tilde{d}_{\alpha},\tilde{d}_\gamma): x^0_\alpha<\tilde{d}_{\alpha}, x^0_\gamma<\tilde{d}_{\gamma}\}(=\Omega_6^a\cup\Omega_6^b$ in Figure~\ref{Fg:Partition}). Finally, in this case, the initial capacity is not enough to satisfy the capacity requirement for either task. Hence, excess demand is lost for both tasks. The second stage cost is
\[
v(x^1_{\alpha},x^1_{\gamma}, \tilde{\xi})=h_{\alpha}(\tilde{d}_\alpha-x^0_\alpha)+h_{\gamma}(\tilde{d}_\gamma-x^0_\gamma).
\]
\end{enumerate}

Now we are ready to state our first result.
\begin{proposition}
If  $\mathbb{P}(h_{\alpha}\geq \tilde{\delta}^1_\gamma h_{\gamma}, c_{\alpha}^2>\tilde{\delta}_{\alpha}^2h_{\alpha},  c_{\gamma}^2>\tilde{\delta}_{\gamma}^2h_{\gamma})=1,$ then,
 \begin{enumerate}
\item  any ${x^1_{\alpha}}^*\in [0,x^0_{\gamma}]$, that satisfies the equation
\begin{equation}
\mathbb{E}(\tilde{\delta^1_\alpha}; x_{\alpha}^0+\tilde{\delta}^1_{\alpha}{x_{\alpha}^1}^*\leq d_{\alpha}, x_{\gamma}^0-{x_{\alpha}^1}^*>d_{\gamma})=\displaystyle\frac{c_{\alpha}^1}{h_{\alpha}},
\label{Eq:Case1Sol1}
\end{equation}
is an optimal cross-training level to work on task $\alpha$.

 \item if  $
\mathbb{E}(\tilde{\delta^1_\alpha}; x_{\alpha}^0+\tilde{\delta}^1_{\alpha}{x_{\alpha}^1}\leq d_{\alpha}, x_{\gamma}^0-{x_{\alpha}^1}>d_{\gamma})\geq\displaystyle\frac{c_{\alpha}^1}{h_{\alpha}}\mbox{ when $x^1_{\alpha}\in[0,x^0_{\gamma}]$}$
then ${x^1_\alpha}^*=x_\gamma^0$ is the optimal cross-training level to work on task $\alpha$.

\item if $
\mathbb{E}(\tilde{\delta^1_\alpha}; x_{\alpha}^0+\tilde{\delta}^1_{\alpha}{x_{\alpha}^1}^*\leq d_{\alpha}, x_{\gamma}^0-{x_{\alpha}^1}^*>d_{\gamma})\leq\displaystyle\frac{c_{\alpha}^1}{h_{\alpha}}\mbox{ on $[0,x^0_{\gamma}]$}$
then ${x^1_\alpha}^*=0$ is the optimal cross-training level to work on task $\alpha$.
\end{enumerate}
Same result holds for ${x^1_\gamma}^*$ when $\alpha$ and $\gamma$ are interchanged above.
\label{Pr:Case1a}

\end{proposition}

\begin{proof}
We derive the first order optimality conditions by setting the derivatives equal to 0. Then we need to check the second order derivatives to ensure convexity. First, we consider the partial derivative with respect to the decision variable $x_\alpha^1$. Since the bounds on productivity factors $\tilde{\delta}^1$ and $\tilde{\delta}^2$ do not depend on the decision variables, using Leibniz rule
\[
\displaystyle\frac{\partial \mathbb{E}[\mathbb{E}[g(x^1_{\alpha},x^1_{\gamma}, \tilde{\xi})|\tilde{\delta}^1, \tilde{\delta}^2]]}{\partial x^1_\alpha}=\mathbb{E}\left[\displaystyle\frac{\partial \mathbb{E}[g(x^1_{\alpha},x^1_{\gamma}, \tilde{\xi})|\tilde{\delta}^1, \tilde{\delta}^2]}{\partial x^1_\alpha}\right]=c_\alpha^1+\mathbb{E}\left[\displaystyle \sum_{i=1}^6\frac{\partial\mathbb{E}[v(x^1_{\alpha},x^1_{\gamma}, \tilde{\xi});\Omega_i|\tilde{\delta}^1, \tilde{\delta}^2]}{\partial x^1_\alpha}\right].
\]
The second stage cost function $v(x^1_{\alpha},x^1_{\gamma}, \tilde{\xi})$ is constant with respect to $x^1_\alpha$ on $\Omega_4,\Omega_5$ and $\Omega_6$ and the bounds of these regions do not depend on $x^1_\alpha$. Also on $\Omega_1$, the second stage cost is uniformly equal to 0. Hence, the derivatives of expectation over these regions are all equal to 0. To simplify the notation, we use $f(\tilde{d}_{\alpha} ,\tilde{d}_{\gamma})$ to denote the density function of demand vector when productivity factors are given. The derivative of expectation over $\Omega_2$ is
\[
\begin{array}{rl}
\displaystyle\frac{\partial\mathbb{E}[v(x^1_{\alpha},x^1_{\gamma}, \tilde{\xi});\Omega_2|\tilde{\delta}^1, \tilde{\delta}^2]}{\partial x^1_\alpha} &=-h_\alpha\tilde{\delta}_\alpha^1\displaystyle\int_{x^0_\alpha+\tilde{\delta}_\alpha^1x_{\alpha}^1}^\infty \int_{0}^{x_\gamma^0-x_{\alpha}^1}f(d_\alpha,d_{\gamma})dd_\gamma dd_{\alpha}\\
&\quad-\displaystyle\int_{x^0_\alpha+\tilde{\delta}_\alpha^1x_{\alpha}^1}^\infty h_\alpha(d_\alpha-x_\alpha^0-\tilde{\delta}_\alpha^1x_{\alpha}^1)f(d_\alpha,x_\gamma^0-x_{\alpha}^1) dd_{\alpha}
\end{array}
\]
Similarly, we can calculate the derivative of expectation over $\Omega_3$
\[
\displaystyle\frac{\partial\mathbb{E}[v(x^1_{\alpha},x^1_{\gamma}, \tilde{\xi});\Omega_3|\tilde{\delta}^1, \tilde{\delta}^2]}{\partial x^1_\alpha}=\displaystyle\int_{x^0_\alpha+\tilde{\delta}_\alpha^1x_{\alpha}^1}^\infty h_\alpha(d_\alpha-x_\alpha^0-\tilde{\delta}_\alpha^1x_{\alpha}^1)f(d_\alpha,x_\gamma^0-x_{\alpha}^1) dd_{\alpha}
\]
Canceling the boundary terms, we get
\begin{equation}\begin{array}{rl}
\displaystyle\frac{\partial \mathbb{E}[\mathbb{E}[g(x^1_{\alpha},x^1_{\gamma}, \tilde{\xi})|\tilde{\delta}^1, \tilde{\delta}^2]]}{\partial x^1_\alpha} 
&=c_\alpha^1-h_\alpha \mathbb{E}[\tilde{\delta^1_\alpha}; x_{\alpha}^0+\tilde{\delta}^1_{\alpha}x_{\alpha}^1\leq d_{\alpha}, x_{\gamma}^0-x_{\alpha}^1>d_{\gamma}].
\end{array}
\label{Eq:Case1FD}
\end{equation}
The partial derivative with respect to $x_\gamma^1$ can be found similarly. Setting these terms to equal 0, we get equation (\ref{Eq:Case1Sol1}). Now, we need to show that the ($x_\alpha^1,x_{\gamma}^1$) pairs that solve these equations, actually minimizes the expected cost by showing that expected cost function is convex in decision variables. Equation~(\ref{Eq:Case1FD}) suggests that the first partial derivative of the expected cost function with respect to $x_\alpha^1$ only depends on $x_\alpha^1$, hence we do not need to consider the cross-partials and the hessian matrix is positive semidefinite if the second partial derivatives with respect to $x_{\alpha}^1$ and $x_{\gamma}^1$ are both nonnegative.
\[
\displaystyle\frac{\partial^2 \mathbb{E}[\mathbb{E}[g(x^1_{\alpha},x^1_{\gamma}, \tilde{\xi})|\tilde{\delta}^1, \tilde{\delta}^2]]}{(\partial x^1_\alpha)^2}=-h_{\alpha}\displaystyle\frac{\partial \mathbb{E}[\tilde{\delta}_\alpha^1\displaystyle\int_{x^0_\alpha+\tilde{\delta}_\alpha^1x_{\alpha}^1}^\infty \int_{0}^{x_\gamma^0-x_{\alpha}^1}f(d_\alpha,d_{\gamma})dd_\gamma dd_{\alpha}]}{\partial x^1_\alpha}.
\]

Observe that if $x^1_\alpha<\bar{x}^1_{\alpha}$, then
\[
\{(\tilde{d}_{\alpha},\tilde{d}_\gamma):x^0_\alpha+\tilde{\delta}^1_\alpha \bar{x}^1_{\alpha}<\tilde{d}_{\alpha}, 0\leq \tilde{d}_{\gamma}\leq x_{\gamma}^0-\bar{x}_{\alpha}^1\}\subseteq \{(\tilde{d}_{\alpha},\tilde{d}_\gamma):x^0_\alpha+\tilde{\delta}^1_\alpha x^1_{\alpha}<\tilde{d}_{\alpha}, 0\leq \tilde{d}_{\gamma}\leq x_{\gamma}^0-x_{\alpha}^1\}.
\]
Using this relation and the fact that the density function and the productivity factor $\tilde{\delta}^1_\alpha$ is always positive, we get
\[
\displaystyle\lim_{\Delta\rightarrow 0}\frac{\mathbb{E}[\tilde{\delta}_\alpha^0(\int_{x^0_\alpha+\tilde{\delta}_\alpha^1(x_{\alpha}^1+\Delta)}^\infty \int_{0}^{x_\gamma^0-(x_{\alpha}^1+\Delta)}f(d_\alpha,d_{\gamma})dd_\gamma dd_{\alpha}-\int_{x^0_\alpha+\tilde{\delta}_\alpha^1x_{\alpha}^1}^\infty \int_{0}^{x_\gamma^0-x_{\alpha}^1}f(d_\alpha,d_{\gamma})dd_\gamma dd_{\alpha})]}{\Delta}\leq 0.
\]
Plugging this back into the second derivative and repeating the same procedure for $x_\gamma^1$, we conclude that the hessian is positive semidefinite and the expected cost function is convex.

When the expectation satisfies the inequality for the second case, the derivative of the cost function is negative for any value in $[0,x_\gamma^0]$. Hence, to minimize the cost, we need to set $x^{1*}_{\alpha}$ to the maximum possible value. The third case can be proven similarly. 

Since the distribution is assumed to be continuous, the derivative in (\ref{Eq:Case1FD}) is continuous in $x_{\alpha}^1$. If the expectation takes both negative and positive values over $x_{\alpha}^1\in[0,x_\gamma^0]$, then intermediate value theorem ensures us that (\ref{Eq:Case1Sol1}) will be satisfied for some $x^{1*}_{\alpha}$. Hence, the three cases stated in the proposition covers all possible situations. 
  \end{proof}

Proposition~\ref{Pr:Case1a} assumes continuous distributions. However, the results can be extended to discrete setting using a similar methodology to the newsvendor problem. 

An important feature of Proposition~\ref{Pr:Case1a} is that it suggests that the problem is separable, i.e., the optimal cross-training levels for tasks $\alpha$ and $\gamma$ can be decided separately. Hence,  we only state corollaries relating to task $\alpha$ below and similar results hold for $\gamma$ by interchanging the subscripts. An immediate consequence of the proposition is as follows.
\begin{corollary}
If the assumption of Proposition~\ref{Pr:Case1a} holds and the first-stage productivity factor $\tilde{\delta}^1_{\alpha}$ is deterministically equal to $\delta^1_\alpha$, then 
\begin{enumerate}
\item  any ${x^1_{\alpha}}^*\in [0,x^0_{\gamma}]$ that satisfies the newsvendor-type equation
\begin{equation}
\mathbb{P}( x_{\alpha}^0+\delta^1_{\alpha}x_{\alpha}^1\leq \tilde{a}_{\alpha}, x_{\gamma}^0-x_{\alpha}^1>\tilde{d}_{\gamma})=\displaystyle\frac{c_{\alpha}^1}{\delta^1_\alpha h_{\alpha}}
\label{Eq:DetDeltaCor}
\end{equation}
solves the cross-training problem.

\item if $\mathbb{P}( x_{\alpha}^0+\delta^1_{\alpha}x_{\alpha}^1\leq \tilde{d}_{\alpha}, x_{\gamma}^0-x_{\alpha}^1>\tilde{d}_{\gamma})\geq \displaystyle\frac{c_{\alpha}^1}{\delta^1_\alpha h_{\alpha}}$, then $x_{\alpha}^{1*}=x_{\gamma}^0$ is the optimal cross-training level to work on task $\alpha$.
    
\item if $\mathbb{P}( x_{\alpha}^0+\delta^1_{\alpha}x_{\alpha}^1\leq \tilde{d}_{\alpha}, x_{\gamma}^0-x_{\alpha}^1>\tilde{d}_{\gamma})\leq \displaystyle\frac{c_{\alpha}^1}{\delta^1_\alpha h_{\alpha}}$, then $x_{\alpha}^{1*}=0$ is the optimal cross-training level to work on task $\alpha$.
    \end{enumerate}
\label{Cr:DetDelta}
\end{corollary}

Proposition~\ref{Pr:Case1a} also helps us understand when it is not profitable to cross-train. 
\begin{corollary}
It is not profitable to cross-train for task $\alpha$, if the assumption of Proposition~\ref{Pr:Case1a} holds and at least one of  $\mathbb{E}[\tilde{\delta}_{\alpha}^1]$, $\mathbb{P}(x_{\alpha}^0\leq \tilde{d}_{\alpha})$ or $\mathbb{P}(x_{\gamma}^0\leq \tilde{d}_{\gamma})$ is less than $\displaystyle\frac{c_{\alpha}^1}{h_{\alpha}}$.
\end{corollary}

\begin{proof}
Using the fact that $\tilde{\delta}_\alpha^1\geq 0$ with probability one, we get
\[
\mathbb{E}[\tilde{\delta^1_\alpha}; x_{\alpha}^0+\tilde{\delta}^1_{\alpha}{x_{\alpha}^1}^*\leq \tilde{d}_{\alpha}, x_{\gamma}^0-{x_{\alpha}^1}^*>\tilde{d}_{\gamma}]\leq\mathbb{E}[\tilde{\delta^1_\alpha}]\leq \frac{c_{\alpha}^1}{h_{\alpha}}.
\]
Using the third part of Proposition~\ref{Pr:Case1a} the result follows. To prove the second and third parts of the corollary, we use the same methodology realizing  $\tilde{\delta}^1_\alpha\leq 1$ with probability one.
  \end{proof}

Intuitively, one should not resort to cross-training if the effectiveness of training programs is not good enough. The first part of the corollary quantifies how ``good enough'' should be understood. The second part says that before cross-training one should make sure that there is a solid chance that extra capacity will be needed. Even when extra capacity is needed, demand for the other task may make it impossible to utilize the cross-trained workers. 

Under some mild conditions, we can use Proposition~\ref{Pr:Case1a} to infer how the variance of the demand for different tasks affect the cross-training decisions. 
\begin{corollary}
Suppose $\tilde{d}_1$ and $\tilde{d}_2$ are symmetric random variables around zero, i.e., $$\mathbb{P}(\tilde{d}_i>x)=\mathbb{P}(\tilde{d}_i<-x)$$ for all values of $x$ and $i=1,2$. If $\tilde{\delta}^1_{\alpha}$ is deterministically equal to $\delta^1_\alpha$, $2c_{\alpha}^1\geq\delta_1^\alpha h_{\alpha}$, the demands for different tasks, $\tilde{d}_{\alpha}$ and $\tilde{d}_{\gamma}$, are independent, continuous and can be written as 
\[\tilde{d}_{\alpha}=m\tilde{d}_1+\mu_1 \mbox{ and } \tilde{d}_{\gamma}=n\tilde{d}_2+\mu_2,\]
then the optimal cross-training level $x_{\alpha}^{1*}$ is non-increasing in both $m$ and $n$.
\label{Cr:Var}
\end{corollary}

\begin{proof}
Using equation (\ref{Eq:DetDeltaCor}) and independence, we get
\[\begin{array}{rl}
\mathbb{P} \left(\displaystyle\frac{x_{\alpha}^0+\delta^1_{\alpha}{x_{\alpha}^1}^*-\mu_1}{m}\leq \tilde{d}_{1}\right) \mathbb{P}\left(\displaystyle\frac{x_{\gamma}^0-{x_{\alpha}^1}^*-\mu_2}{n}>\tilde{d}_{2}\right)&=\displaystyle\frac{c_{\alpha}^1}{\delta_1^\alpha h_{\alpha}}\geq \displaystyle \frac{1}{2}.
\end{array}
\]
Now, we can infer that the probabilities on the left-hand side of the inequality should be greater than 0.5. Then, using the fact that $\tilde{d}_1$ and $\tilde{d}_2$ are symmetric random variables
\begin{equation}
\frac{x_{\alpha}^0+\delta^1_{\alpha}{x_{\alpha}^1}^*-\mu_1}{m}\leq 0 \mbox{ and } \frac{x_{\gamma}^0-{x_{\alpha}^1}^*-\mu_2}{n}\geq 0.
\label{Eq:moninequal}
\end{equation}
The left-hand side of equation (\ref{Eq:DetDeltaCor}) decreases as $m$ increases. Hence, if statement 1 of Corollary~\ref{Cr:DetDelta} is true, we need to decrease $x^{1*}_{\alpha}$ to recover the equality. If statement 2 is true, we may either wish to stay at $x^0_{\gamma}$ or we may wish to decrease $x^1_{\alpha}$. For the third statement, we do not need to take any action. Hence, this proves that the optimal cross-training level is non-increasing in $m$. Similar arguments show that $x^{1*}_{\alpha}$ is non-increasing in $n$.
  \end{proof}

Corollary~\ref{Cr:Var} essentially states that if demands for products are independent and follow a symmetric distribution, e.g. normal or uniform distributon, and if the costs and productivity factors satisfy the conditions above, the optimal cross-training level for task $\alpha$ is decreasing with the variances of demands. Unfortunately, when $\tilde{\delta}_\alpha^1$ is not deterministic, we can construct counter-examples where the monotonicity result does not hold. When $\tilde{\delta}_\alpha^1$ is random, the first inequality of (\ref{Eq:moninequal}) may fail to hold for some realizations of $\tilde{\delta}_\alpha^1$ and these may force us to increase the cross-training level as the variance increases. In Section~\ref{Sec:NumRes}, we provide examples to show that we can lose monotonicity, if $2c_{\alpha}^1<\tilde{\delta}_1^\alpha h_{\alpha}$. The next corollary shows how the cross-training policies are affected by the variances of first-stage productivity factors.

\begin{corollary}
Suppose $\tilde{\delta}$ is a symmetric random variable around zero with support $\Omega_{\delta}$, and 
$\tilde{\delta}_\alpha^1=\Delta+n\tilde{\delta}$. If $\mathbb{P}(h_{\alpha}\geq \tilde{\delta}^1_\gamma h_{\gamma}, c_{\alpha}^1>\tilde{\delta}_{\alpha}^2h_{\alpha},  c_{\gamma}^1>\tilde{\delta}_{\gamma}^2h_{\gamma})=1$, and  $\tilde{d}_\alpha$ and $\tilde{d}_{\gamma}$ are independent, then the optimal cross-training level $x^{1*}$ is non-increasing in $n$. 
\end{corollary}
 
\begin{proof} Using independence, we can write (\ref{Eq:Case1Sol1}) as
\begin{equation}
\left(\int_{\Omega_{\delta}}\int_{x^0_\alpha+(\Delta+n\delta)x^{1*}_{\alpha}}^\infty(\Delta+n\delta)f_{\tilde{d}_{\alpha},\tilde{\delta}}(d_{\alpha},\delta)dd_\alpha d\delta\right)\mathbb{P}(x_{\gamma}^0-x^{1*}_{\alpha}<\tilde{d}_{\gamma})=\frac{c^1_{\alpha}}{h_{\alpha}}.
\label{Eq:FSvarinc}
\end{equation}
First, we note that both the first and second multiplier on the left-hand side are non-increasing in $x^{1*}_{\alpha}$ for any value of $n$. Taking the derivative of the first multiplier on the left-hand side with respect to $n$, we get
\[\begin{array}{l}
\displaystyle\frac{d\left(\int_{\Omega_{\delta}}\int_{x^0_\alpha+(\Delta+n\delta)x^{1*}_{\alpha}}^\infty(\Delta+n\delta)f_{\tilde{d}_{\alpha}, \tilde{\delta}}(d_{\alpha},\delta)dd_{\alpha}d\delta\right)}{dn}=\displaystyle \int_{0}^\infty\int_{\Omega_{\delta}\cap\{\delta<\frac{d_\alpha-x^0_\alpha-x^{1*}_{\alpha}\Delta}{x^{1*}n}\}}\delta f_{\tilde{d}_{\alpha}, \tilde{\delta}}(d_{\alpha},\delta)d\delta dd_{\alpha}\\
\quad\quad\quad\quad\quad\quad\quad\quad\quad\quad\quad\quad\quad\quad\quad\quad\quad\quad\quad\displaystyle -\int_{\Omega_{\delta}}x^{1*}_{\alpha}\delta(\Delta+n\delta)f_{\tilde{d}_{\alpha}, \tilde{\delta}}(x^0_\alpha+(\Delta+n\delta)x^{1*}_{\alpha},\delta)d\delta.
\end{array}
\]
Using $\mathbb{E}[\tilde{\delta}]=0$, we can infer that the first term is non-positive. Also, as all the integrands in the second term are non-negative, we can conclude that the derivative is non-positive. Thus, the first multiplier in (\ref{Eq:FSvarinc}) is non-increasing with respect to $n$ and if it decreases as $n$ increases, we need to decrease $x^{1*}$ to restore the equality in (\ref{Eq:FSvarinc}).   
\end{proof}

We construct the proof based on the assumption that cross-training will not be beneficial in the second stage. However, if the second-stage productivity is preferable, then we might lose the monotonicity of the first-stage cross-training with respect to the variance of the productivity factor. Examples of such cases are demonstrated in Section~\ref{Sec:NumRes}.
\vspace{-0.1in}

\subsubsection{Case 1.b: $c_{\alpha}^2\leq \tilde{\delta}_{\alpha}^2h_{\alpha}$ and $ c_{\gamma}^2\leq \tilde{\delta}_{\gamma}^2h_{\gamma}$.}
\vspace{-0.1in}

Now, we consider the situation where it is profitable to cross-train in the second stage. The demand is lost only when the available workforce is not able to satisfy the demand even after cross-training all the free workers. Hence, the cost function will differ only when there is a trade-off between the second stage cross-training and losing demand, i.e., the cost function does not change for $\Omega_1, \Omega_3,\Omega_4$ and $\Omega_6$ and we only need to consider $\Omega_2$ and $\Omega_5$ further. 

\begin{enumerate}

\item $\Omega_2^a=\{(\tilde{d}_{\alpha},\tilde{d}_\gamma): x^0_\alpha+\tilde{\delta}^1_\alpha x^1_{\alpha}<\tilde{d}_{\alpha}\leq x^0_\alpha+\tilde{\delta}^1_\alpha x^1_{\alpha}+\tilde{\delta}^2_\alpha(x^0_\gamma-x^1_{\alpha}-\tilde{d}_{\gamma}), \tilde{d}_{\gamma}< x^0_\gamma-x^1_{\alpha}\}$. If $(\tilde{d}_{\alpha},\tilde{d}_\gamma)\in \Omega_2^a$, both the initial workforce and the cross-trained workforce are used in performing task $\alpha$. If $(\tilde{d}_{\alpha}-x^0_\alpha-\tilde{\delta}^1_\alpha x^1_{\alpha})/\tilde{\delta}^2_\alpha$ units of the workforce are cross-trained from task $\gamma$, the remaining demand can be satisfied. Hence, the second stage cross-training cost is \[v(x^1_{\alpha},x^1_{\gamma}, \tilde{\xi})=c^2_\alpha\frac{\tilde{d}_{\alpha}-x^0_\alpha-\tilde{\delta}^1_\alpha x^1_{\alpha}}{\tilde{\delta}^2_\alpha}.\]

\item $\Omega_2^b=\{(\tilde{d}_{\alpha},\tilde{d}_\gamma): x^0_\alpha+\tilde{\delta}^1_\alpha x^1_{\alpha}+\tilde{\delta}^2_\alpha(x^0_{\gamma}-x^1_{\alpha}-\tilde{d}_{\gamma})<\tilde{d}_{\alpha}, \tilde{d}_{\gamma}< x^0_\gamma-x^1_{\alpha}\}$. For the scenarios in $\Omega_2^b$, it is not possible to satisfy the demand for $\alpha$ even after all the cross-trained workforce  work on task $\alpha$. The decision maker cross-trains the idle workforce of $\gamma$ and then excess demand is lost. Hence, over $\Omega_2^b$ 
\[
v(x^1_{\alpha},x^1_{\gamma}, \tilde{\xi})=c^2_\alpha(x^0_\gamma-x^1_\alpha-\tilde{d}_\gamma)+h_\alpha(\tilde{d}_\alpha-x_{\alpha}-\tilde{\delta}^1_\alpha x^1_{\alpha}-\tilde{\delta}^2_\alpha(x^0_{\gamma}-x^1_{\alpha}-\tilde{d}_{\gamma})).
\]
\end{enumerate}
The structure for $\Omega_5^a$ and $\Omega_5^b$ is similar and we are ready to state the allocations for this case.
\begin{proposition}
If $\mathbb{P}(h_{\alpha}\geq \tilde{\delta}^1_\gamma h_{\gamma}, c_{\alpha}^2\leq\tilde{\delta}_{\alpha}^2h_{\alpha},  c_{\gamma}^2\leq\tilde{\delta}_{\gamma}^2h_{\gamma})=1,$ and the training effectiveness for programs before and after demand is realized are the same, i.e., $\tilde{\delta}^1_\alpha=\tilde{\delta}^2_\alpha$  with probability one, then optimal cross-training decisions do not depend on the opportunity cost $h_\alpha$ and 
\begin{enumerate}
\item any solution ${x^1_{\alpha}}^*$ that satisfies the newsvendor-type equation
\begin{equation}
\mathbb{P}(\tilde{d}_\alpha\geq x^0_\alpha+\tilde{\delta}^1_\alpha {x^1_{\alpha}}^*, \tilde{d}_\gamma \leq x^0_\gamma-{x^1_{\alpha}}^*)=\displaystyle\frac{c^1_\alpha}{c^2_\alpha}
\label{Eq:Newsvendor}
\end{equation}
is an optimal cross-training level to work on task $\alpha$.

\item if $\mathbb{P}(\tilde{d}_\alpha\geq x^0_\alpha+\tilde{\delta}^1_\alpha x^0_\gamma, \tilde{d}_\gamma= 0)\geq\displaystyle\frac{c^1_\alpha}{c^2_\alpha}$, then $x_\alpha^{1*}=x_\gamma^0$ is the optimal cross-training level to work on task $\alpha$.
    
\item if $\mathbb{P}(\tilde{d}_\alpha\geq x^0_\alpha, \tilde{d}_\gamma= x_\gamma^0)\leq \displaystyle\frac{c^1_\alpha}{c^2_\alpha}$, then $x_\alpha^{1*}=0$ is the optimal cross-training level to work on task $\alpha$.
\end{enumerate}
The same result holds for task $\gamma$ with $\alpha$ and $\gamma$ interchanged above.
\label{Pr:Case1b}
\end{proposition}
\begin{proof}
Similar to the proof of Proposition~\ref{Pr:Case1a} we need to derive the first and second order optimality conditions. On regions $\Omega_1,\Omega_3, \Omega_4$ and $\Omega_6$, the structure is the same as in Proposition~\ref{Pr:Case1a} and the problem is separable in cross-training levels $x_\alpha^1$ and $x_\gamma^1$.

The derivative of the second-stage cost function over $\Omega_2^a$ can be calculated as:
\[
\begin{array}{l}
\displaystyle\frac{\partial\mathbb{E}[v(x^1_{\alpha},x^1_{\gamma}, \tilde{\xi});\Omega_2^a|\tilde{\delta}^1, \tilde{\delta}^2]}{\partial x^1_\alpha}=
\displaystyle-\int_0^{x^0_\gamma-x^1_\alpha}\int_{x^0_\alpha+\tilde{\delta}^1_\alpha x^1_{\alpha}}^{x^0_\alpha+\tilde{\delta}^1_\alpha x^1_{\alpha}+\tilde{\delta}^2_\alpha(x^0_\gamma-x^1_{\alpha}-\tilde{d}_{\gamma})} \frac{c^2_\alpha\tilde{\delta}^1_\alpha}{\tilde{\delta}^2_\alpha}f(\tilde{d}_{\alpha} ,\tilde{d}_{\gamma})d\tilde{d}_{\alpha} d\tilde{d}_{\gamma}\\
\quad\quad\quad\quad\quad\quad\quad\displaystyle+\int_0^{x^0_\gamma-x^1_\alpha}(\tilde{\delta}^1_\alpha-\tilde{\delta}^2_\alpha)c^2_\alpha(x^0_\gamma-x^1_{\alpha}-\tilde{d}_{\gamma})f(x^0_\alpha+\tilde{\delta}^1_\alpha x^1_{\alpha}+\tilde{\delta}^2_\alpha(x^0_\gamma-x^1_{\alpha}-\tilde{d}_{\gamma}),\tilde{d}_\gamma)d\tilde{d_\gamma}.
\end{array}
\]
The derivative over $\Omega_2^b$ is
\[
\begin{array}{l}
\displaystyle\frac{\partial\mathbb{E}[v(x^1_{\alpha},x^1_{\gamma}, \tilde{\xi});\Omega_2^b|\tilde{\delta}^1, \tilde{\delta}^2]}{\partial x^1_\alpha}=
\displaystyle-\int_0^{x^0_\gamma-x^1_{\alpha}}\int_{x^0_\alpha+\tilde{\delta}^1_\alpha x^1_{\alpha}+\tilde{\delta}^2_\alpha(x^0_{\gamma}-x^1_{\alpha}-\tilde{d}_{\gamma})}^\infty (c^2_\alpha+h_\alpha(\tilde{\delta}^1_\alpha-\tilde{\delta}^2_\alpha)) f(\tilde{d}_{\alpha} ,\tilde{d}_{\gamma})d\tilde{d}_{\alpha} d\tilde{d}_{\gamma}\\
\quad\quad\quad\quad\quad\quad\quad\quad\displaystyle-\int_0^{x^0_\gamma-x^1_\alpha}(\tilde{\delta}^1_\alpha-\tilde{\delta}^2_\alpha)c^2_\alpha(x^0_\gamma-x^1_{\alpha}-\tilde{d}_{\gamma})f(x^0_\alpha+\tilde{\delta}^1_\alpha x^1_{\alpha}+\tilde{\delta}^2_\alpha(x^0_\gamma-x^1_{\alpha}-\tilde{d}_{\gamma}),\tilde{d}_\gamma)d\tilde{d}_\gamma\\
\quad\quad\quad\quad\quad\quad\quad\quad\displaystyle-\int_{x^0_\alpha+\tilde{\delta}^1_\alpha x^1_{\alpha}}^\infty h_\alpha(\tilde{d}_{\alpha}-(x^0_\alpha+\tilde{\delta}^1_\alpha x^1_\alpha))f(\tilde{d}_{\alpha} ,x^0_\gamma-x^1_\alpha)d\tilde{d}_{\alpha}.
\end{array}
\]
Aggregating all the results and cancelling the appropriate terms we get
\begin{equation}
\begin{array}{ll}
\displaystyle \frac{\partial \mathbb{E}[g(x^1_{\alpha},x^1_{\gamma}, \tilde{\xi})|\tilde{\delta}^1, \tilde{\delta}^2]}{\partial x^1_\alpha}&\displaystyle=c^1_\gamma\displaystyle-\int_0^{x^0_\gamma-x^1_\gamma}\int_{x^0_\alpha+\tilde{\delta}^1_\gamma x^1_{\gamma}}^{x^0_\alpha+\tilde{\delta}^1_\gamma x^1_{\gamma}+\tilde{\delta}^2_\gamma(x^0_\gamma-x^1_{\gamma}-\tilde{d}_{\gamma})} \frac{c^2_\alpha\tilde{\delta}^1_\alpha}{\tilde{\delta}^2_\alpha}f(\tilde{d}_{\alpha} ,\tilde{d}_{\gamma})d\tilde{d}_{\alpha} d\tilde{d}_{\gamma}\\
&\displaystyle\quad-\int_0^{x^0_\gamma-x^1_{\gamma}}\int_{x^0_\alpha+\tilde{\delta}^1_\gamma x^1_{\gamma}+\tilde{\delta}^2_\gamma(x^0_{\gamma}-x^1_{\gamma}-\tilde{d}_{\gamma})}^\infty (c^2_\alpha+h_\alpha(\tilde{\delta}^1_\gamma-\tilde{\delta}^2_\gamma)) f(\tilde{d}_{\alpha} ,\tilde{d}_{\gamma})d\tilde{d}_{\alpha} d\tilde{d}_{\gamma}.
\end{array}
\label{Eq:First_Result}
\end{equation}
If the training effectiveness is the same before and after the demand is observed, i.e., $\tilde{\delta}^1_\alpha=\tilde{\delta}^2_\alpha$, then the first order condition  reduces to (\ref{Eq:Newsvendor}) and the convexity of the cost function is proven similar to Proposition~\ref{Pr:Case1a}.  Since expectation preserves convexity, expected cost  is convex and the solution which satisfies (\ref{Eq:Newsvendor}) solves the cross-training problem for two tasks.
  \end{proof}

Under scenarios corresponding to Case 1.b the opportunity cost is incurred after all the workforce is assured to be working on a task, i.e., only for the demand which can not be satisfied from the available effective capacity. By analyzing the right hand side of (\ref{Eq:First_Result}), we see that if $\tilde{\delta}^1_\alpha$ and $\tilde{\delta}^2_\alpha$ are fixed, the marginal opportunity cost incurred by not cross-training one unit of capacity in the first stage is given by
\[
h_{\alpha}(\tilde{\delta}^1_\alpha-\tilde{\delta}^2_\alpha)\mathbb{P}(x^0_\alpha+\tilde{\delta}^1_\alpha x^1_{\alpha}+\tilde{\delta}^2_\alpha(x^0_{\gamma}-x^1_{\alpha}-\tilde{d}_{\gamma})<\tilde{d}_{\alpha}, \tilde{d}_{\gamma}\leq x^0_\gamma-x^1_{\alpha}).
\]
The opportunity cost plays a role just because of the scenarios in $\Omega^b_2$, i.e., for the scenarios where the second stage cross-training is needed. To understand this further, suppose we defer training one unit of workforce to the second stage. This means that this workforce works with effective capacity of $\tilde{\delta}^2_\alpha$, instead of $\tilde{\delta}^1_\alpha$, i.e. we lose $\tilde{\delta}^1_\alpha-\tilde{\delta}^2_\alpha$ units of capacity by training this workforce in the second stage rather than the first stage. Hence, only the opportunity cost for this lost capacity affects our decisions. If the training effectiveness in both stages is the same, then essentially we do not lose any available effective capacity and hence opportunity cost does not play any role in our decisions.

Another question is what happens when $\tilde{\delta}^1_\alpha \neq \tilde{\delta}^2_\alpha$. We see that equation~(\ref{Eq:First_Result}) would reduce the cross-training problem to solving a simple univariate nonlinear equation if we were to show that the function is convex over the feasible region $[0,x_{\gamma}^0]$. Unfortunately, it can be shown that the objective function is not convex with respect to the decision variables in general and this non-convexity depends on the probability distribution of the demand vector. However, we can state the following relationship between the total cost  second stage productivity factors, even when non-convexity is present.

\begin{proposition}
The total cost is a convex function of $\delta^2_{\alpha}$ over the interval $[c_{\alpha}^2/h_{\alpha},1]$.
\label{Pr:Convdelta}
\end{proposition}
\begin{proof}
The cost related to the demand for task $\gamma$ is a constant with respect to $\delta^2_\alpha$ and for ease of notation we denote it as $C$. Then, given a demand scenario $d$ and first-stage decisions $x^1$, the second stage cost can be written as a function of $\tilde{\delta}_\alpha^2$ as follows:
\[
\begin{array}{rl}
v(x^1, d, \delta^1, \delta^2)&=\max\left\{0, c^2_\alpha\displaystyle\frac{d_{\alpha}-x^0_\alpha-\delta^1_\alpha x^1_{\alpha}}{\delta^2_\alpha},\right.\\
&\quad\quad\quad\quad\left.c^2_\alpha(x^0_\gamma-x^1_\alpha-d_\gamma)+h_\alpha(d_\alpha-x_{\alpha}-\delta^1_\alpha x^1_{\alpha}-\delta^2_\alpha(x^0_{\gamma}-x^1_{\alpha}-d_{\gamma}))\right\}+C
\end{array}
\]
The second stage cost function is the maximum of convex functions of $\delta^2_\alpha$ and hence it is also convex. As expectation preserves convexity, the result follows.  
  \end{proof}

The convexity properties of the cost with respect to random variables can be used to infer how the variances of these variables affect the total cost. Proposition~\ref{Pr:Convdelta} essentially indicates that as long as the mean of $\tilde{\delta}^2_\alpha$ is kept constant and $\mathbb{P}(c_{\alpha}^2\leq \tilde{\delta}_{\alpha}^2h_{\alpha})=1$, the total cost increases as the variance of $\tilde{\delta}^2_\alpha$ increases. The condition $\mathbb{P}(c_{\alpha}^2\leq \tilde{\delta}_{\alpha}^2h_{\alpha})=1$ is of crucial importance for this interpretation. If there exists a $\delta^*$ in the support of $\tilde{\delta}^2_\alpha$ such that $c_{\alpha}^2>\delta^*h_{\alpha}$, the decision maker prefers incurring opportunity cost when $\tilde{\delta}^2_{\alpha}=\delta^*$. Hence, in a neighborhood of $\delta^*$ the total cost is constant with respect to $\delta^2_{\alpha}$. If $0<\mathbb{P}(c_{\alpha}^2\leq \tilde{\delta}_{\alpha}^2h_{\alpha})<1$, the decision maker decided whether to lose the demand or cross-training in the second stage, i.e., the second-stage cost function is the minimum of two convex function which is not convex in general. This nonconvexity can sometimes cause the cost to decrease as the variance of $\tilde{\delta}^2_\alpha$ increases. Examples for such cases are presented in Section~\ref{Sec:NumRes}. 

In both Cases 1.a and 1.b, the decisions for tasks can be made separately. Hence, we can use the results above, when it is profitable to cross-train in the second stage for one task and hire without cross-training for the other task, once we make sure that  $h_\alpha\geq \tilde{\delta}^1_\gamma h_{\gamma}$.
\vspace{-0.2in} 

\subsection{Case 2: Workers Used in Their New Tasks $(h_\alpha\leq \tilde{\delta}^1_\gamma h_{\gamma})$}
\vspace{-0.1in}

In this section, we analyze the cases where it is profitable to allocate cross-trained workers to task $\gamma$ even when there is need for task $\alpha$. Due to a similar reasoning as  above, the problem is again separable into tasks. The allocation policy for the workers who are cross-trained to work on task $\alpha$ does not differ from Case 1 under this new assumption. Hence, the optimal cross-training level for task $\alpha$, $x^1_\alpha$ can be decided by using the results above. However, the situation for task $\gamma$ gets more complicated and we have different conditions for $x^{1*}_{\gamma}$. As in Case 1, we base our analysis on whether second stage cross-training is profitable or not. 
\vspace{-0.1in}

\subsubsection{Case 2.a: $c^2_\alpha>\tilde{\delta}^2_\alpha h_\alpha$ and $c^2_\gamma>\tilde{\delta}_\gamma^2h_\gamma$.}
\vspace{-0.1in}

In this case, one should avoid cross-training in the second stage. The conditions are similar to those of Case 1.a, and the second stage cost function $v(x_\alpha^1,x_\gamma^1, \tilde{\xi})$ takes the same form as in Case 1.a over $\Omega_1, \Omega_2, \Omega_3$ or $\Omega_5$. We need to partition regions $\Omega_4$ and $\Omega_6$ further in order to express the second stage cost function in closed form.  

\begin{enumerate}
 \item $\Omega_4^a\cup \Omega_6^b=\{(\tilde{d}_{\alpha},\tilde{d}_\gamma), x^0_\alpha-x^1_{\gamma}<\tilde{d}_{\alpha}\leq x^0_\alpha, x^0_\gamma+\tilde{\delta}^1_\gamma (x^0_{\alpha}-\tilde{d}_{\alpha})<\tilde{d}_{\gamma}<x_\gamma^0+\tilde{\delta}_\gamma^1 x_\gamma^1\}$. Some of the workers cross-trained to work on $\gamma$ are needed for both tasks over this region. As it is more profitable to employ them on task $\gamma$,  task $\alpha$ demand is lost for the shifted workforce. Hence, the second stage cost is given by
\[
v(x_\alpha^1,x_\gamma^1, \tilde{\xi})=h_\alpha \left(\tilde{d}_\alpha-x_\alpha^0+\frac{\tilde{d}_\gamma-x_\gamma^0}{\tilde{\delta}^1_\gamma}\right).
\]

\item $\Omega_4^b\cup \Omega_6^b=\{(\tilde{d}_{\alpha},\tilde{d}_\gamma), x^0_\alpha-x^1_{\gamma}<\tilde{d}_{\alpha}, x_\gamma^0+\tilde{\delta}_\gamma^1 x_\gamma^1<\tilde{d}_{\gamma}\}$. For scenarios in this region, all the workforce cross-trained to work on $\gamma$ are needed for both tasks and will be employed on task $\gamma$. However, even after the workforce is shifted there is excess demand for task $\gamma$. Hence, opportunity cost is incurred for both $\alpha$ and $\gamma$.  The second stage cost function is the sum of these costs and can be calculated as
\[
v(x_\alpha^1,x_\gamma^1, \tilde{\xi})=h_\alpha(\tilde{d}_\alpha-x_\alpha^0+x_\gamma^1)+h_\gamma(\tilde{d}_\gamma-x_\gamma^0-\tilde{\delta}^1_\gamma x_\gamma^1).
\]
\end{enumerate}

\begin{proposition}
If $\mathbb{P}(h_\alpha<\tilde{\delta}^1_\gamma h_\gamma,c^2_\alpha>\tilde{\delta}^2_\alpha, c^2_\gamma>\tilde{\delta}_\gamma^2h_\gamma)=1$, then 
\begin{enumerate}
 \item any $x_\gamma^{1*}\in [0, x_\alpha^0]$ satisfying the equation
\begin{equation}
\mathbb{E}[\tilde{\delta}^1_\gamma h_\gamma;x_\gamma^0+\tilde{\delta}^1_\gamma x_\gamma^1<\tilde{d}_\gamma] -h_\alpha\mathbb{P}(x_\alpha^0-x_\gamma^1<\tilde{d}_\alpha,x_\gamma^0+\tilde{\delta}^1_\gamma x_\gamma^1<\tilde{d}_\gamma)=c_\gamma^1
\label{Eq:Case2a}
\end{equation}
is an optimal cross-training level to work on task $\gamma$.

\item if $\mathbb{E}[\tilde{\delta}^1_\gamma h_\gamma;x_\gamma^0+\tilde{\delta}^1_\gamma x_\gamma^1<\tilde{d}_\gamma] -h_\alpha\mathbb{P}(x_\alpha^0-x_\gamma^1<\tilde{d}_\alpha,x_\gamma^0+\tilde{\delta}^1_\gamma x_\gamma^1<\tilde{d}_\gamma)\geq c_\gamma^1,$ when $x_\gamma^1\in [0, x_\alpha^0]$, then $x_\gamma^{1*}=x_\alpha^0$ is the optimal cross training level to work on task $\gamma$. 

\item if $\mathbb{E}[\tilde{\delta}^1_\gamma h_\gamma;x_\gamma^0+\tilde{\delta}^1_\gamma x_\gamma^1<\tilde{d}_\gamma] -h_\alpha\mathbb{P}(x_\alpha^0-x_\gamma^1<\tilde{d}_\alpha,x_\gamma^0+\tilde{\delta}^1_\gamma x_\gamma^1<\tilde{d}_\gamma)\leq c_\gamma^1,$ when $x_\gamma^1\in [0, x_\alpha^0]$, then $x_\gamma^{1*}=0$ is the optimal cross training level to work on task $\gamma$. 
 
\end{enumerate}

Cross-training level to work on task $\alpha$, $x^{1*}_{\alpha}$ is decided using  Proposition~\ref{Pr:Case1a}.
\end{proposition}

\begin{proof}
We know from the proof of Proposition~\ref{Pr:Case1a} that 
\[
\displaystyle\frac{\partial E(v(x_\alpha^1,x_\gamma,\tilde{\xi});\Omega_1|\tilde{\delta}^1, \tilde{\delta}^2)}{\partial x_\gamma^1}=\displaystyle\frac{\partial E(v(x_\alpha^1,x_\gamma,\tilde{\xi});\Omega_2|\tilde{\delta}^1, \tilde{\delta}^2)}{\partial x_\gamma^1}=\displaystyle\frac{\partial E(v(x_\alpha^1,x_\gamma,\tilde{\xi});\Omega_3|\tilde{\delta}^1, \tilde{\delta}^2)}{\partial x_\gamma^1}=0
\]
and
\[\begin{array}{rl}
\displaystyle\frac{\partial E(v(x_\alpha^1,x_\gamma,\tilde{\xi});\Omega_5|\tilde{\delta}^1, \tilde{\delta}^2)}{\partial x_\gamma^1}&=\displaystyle-\tilde{\delta}^1_\gamma h_\gamma \int_0^{x_\alpha^0-x_\gamma^1}\int_{x_\gamma^0+\tilde{\delta}_\gamma^1x_\gamma^1}^\infty f(d_\alpha,d_\gamma)dd_\gamma dd_\alpha\\&\displaystyle\quad-h_\gamma\int_{x_\gamma^0+\tilde{\delta}_\gamma^1x_\gamma^1}^\infty
(d_\gamma-x_\gamma^0-\tilde{\delta}_\gamma^1x_\gamma^1) f(x_\alpha^0-x_\gamma^1,d_\gamma)dd_\gamma.
\end{array}\]

Using Leibniz rule successively as in previous proofs we get
\[\displaystyle\frac{\partial E(v(x_\alpha^1,x_\gamma,\tilde{\xi});\Omega_4^a\cup \Omega_6^a|\tilde{\delta}^1, \tilde{\delta}^2)}{\partial x_\gamma^1}=\tilde{\delta}_\gamma^1h_\alpha\int_{x_\alpha^0-x_\gamma^1}(d_\alpha-x_\alpha^0-x_\gamma^1)f(d_\alpha,x_\gamma^0+\tilde{\delta}_\gamma^1x_\gamma^1)dd_\alpha\]
and 

\[\begin{array}{rl}
\displaystyle\frac{\partial E(v(x_\alpha^1,x_\gamma,\tilde{\xi});\Omega_4^b\cup \Omega_6^b|\tilde{\delta}^1, \tilde{\delta}^2)}{\partial x_\gamma^1}&=\displaystyle(h_\alpha-\tilde{\delta}^1_\gamma h_\gamma) \int_{x_\alpha^0-x_\gamma^1}^\infty\int_{x_\gamma^0+\tilde{\delta}_\gamma^1x_\gamma^1}^\infty f(d_\alpha,d_\gamma)dd_\gamma dd_\alpha\\&\displaystyle\quad+h_\gamma\int_{x_\gamma^0+\tilde{\delta}_\gamma^1x_\gamma^1}^\infty
(d_\gamma-x_\gamma^0-\tilde{\delta}_\gamma^1x_\gamma^1)f(x_\alpha^0-x_\gamma^1, d_\gamma) dd_\gamma\\&\displaystyle\quad-\tilde{\delta}_\gamma^1h_\alpha\int_{x_\alpha^0-x_\gamma^1}(d_\alpha-x_\alpha^0-x_\gamma^1)f(d_\alpha,x_\gamma^0+\tilde{\delta}_\gamma^1x_\gamma^1)dd_\alpha.
\end{array}\]
Summing the components above, we obtain 
\begin{equation}
\begin{array}{rl}
\displaystyle\frac{\partial \mathbb{E}[g(x^1_\alpha,x^1_\gamma,\tilde{\xi})]}{\partial x_\gamma^1}&=c^1_\gamma-\mathbb{E}[\tilde{\delta}^1_\gamma h_\gamma;x_\gamma^0+\tilde{\delta}^1_\gamma x_\gamma^1<\tilde{d}_\gamma] +h_\alpha\mathbb{P}(x_\alpha^0-x_\gamma^1<\tilde{d}_\alpha,x_\gamma^0+\tilde{\delta}^1_\gamma x_\gamma^1<\tilde{d}_\gamma)
\label{Eq:Case2afir}
\end{array}
\end{equation}
and equation~(\ref{Eq:Case2a}) follows as the first order condition for optimality. Now we prove the convexity of the second stage cost function. The cross-partial derivatives are zero and from Proposition~\ref{Pr:Case1a}, the second partial with respect to $x_\alpha^1$ is positive, hence we only need to check the second partial with respect to $x_\gamma^1$. Without loss of generality, assume $\Delta>0$, then
\begin{equation}
\displaystyle\frac{\partial^2 \mathbb{E}[g(x^1_\alpha,x^1_\gamma,\tilde{\xi})]}{(\partial x_\gamma^1)^2}=\displaystyle\lim_{\Delta\rightarrow 0}\frac{\displaystyle\left.\frac{\partial \mathbb{E}[g(x^1_\alpha,x^1_\gamma,\tilde{\xi})]}{\partial x_\gamma^1}\right|_{x_\gamma^1=s+\Delta}-\left.\frac{\partial \mathbb{E}[g(x^1_\alpha,x^1_\gamma,\tilde{\xi})]}{\partial x_\gamma^1}\right|_{x_\gamma^1=s}}{\Delta} 
\label{Eq:Case2asec}\end{equation}
Then we get,
\[
\mathbb{E}[\tilde{\delta}^1_\gamma h_\gamma;x_\gamma^0+\tilde{\delta}^1_\gamma s+\Delta<\tilde{d}_\gamma]-\mathbb{E}[\tilde{\delta}^1_\gamma h_\gamma;x_\gamma^0+\tilde{\delta}^1_\gamma s<\tilde{d}_\gamma]=\mathbb{E}_{\tilde{\delta}}[-\int_{x_\gamma^0+\tilde{\delta}^1_\gamma s}^{x_\gamma^0+\tilde{\delta}^1_\gamma s+\Delta}\int_0^\infty \tilde{\delta}^1_\gamma h_\gamma f(d_\alpha, d_\gamma)dd_\alpha dd_\gamma]
\]
and
\[
\begin{array}{l}
h_\alpha(\mathbb{P}(x_\gamma^0-s-\Delta<\tilde{d}_\alpha,x_\gamma^0+\tilde{\delta}^1_\gamma s+\Delta<\tilde{d}_\gamma)- \mathbb{P}(x_\gamma^0-s<\tilde{d}_\alpha,x_\gamma^0+\tilde{\delta}^1_\gamma s<\tilde{d}_\gamma))\\
\quad\geq h_\alpha(\mathbb{P}(x_\gamma^0-s-\Delta<\tilde{d}_\alpha,x_\gamma^0+\tilde{\delta}^1_\gamma s+\Delta<\tilde{d}_\gamma)- \mathbb{P}(x_\gamma^0-s-\Delta<\tilde{d}_\alpha,x_\gamma^0+\tilde{\delta}^1_\gamma s<\tilde{d}_\gamma))\\
\quad=-h_\alpha\mathbb{P}(x_\gamma^0-s-\Delta<\tilde{d}_\alpha,x_\gamma^0+\tilde{\delta}^1_\gamma s<\tilde{d}_\gamma\leq x_\gamma^0+\tilde{\delta}^1_\gamma s+\Delta)\\
\quad\geq\displaystyle \mathbb{E}_{\tilde{\delta}}[-\int_{x_\gamma^0+\tilde{\delta}^1_\gamma s}^{x_\gamma^0+\tilde{\delta}^1_\gamma s+\Delta}\int_0^\infty h_\alpha f(d_\alpha, d_\gamma)dd_\alpha dd_\gamma].
\end{array}
\]
Using equations (\ref{Eq:Case2afir}) and (\ref{Eq:Case2asec}) along with above calculations, we get
\[
\displaystyle\frac{\partial^2 \mathbb{E}[g(x^1_\alpha,x^1_\gamma,\tilde{\xi})]}{(\partial x_\gamma^1)^2}\geq\displaystyle\lim_{\Delta\rightarrow 0}\frac{\mathbb{E}_{\tilde{\delta}}[\int_{x_\gamma^0+\tilde{\delta}^1_\gamma s}^{x_\gamma^0+\tilde{\delta}^1_\gamma s+\Delta}\int_0^\infty (\tilde{\delta}^1_\gamma h_\gamma-h_\alpha) f(d_\alpha, d_\gamma)dd_\alpha dd_\gamma]}{\Delta}\geq 0,
\]
where the last inequality follows from the assumption $\tilde{\delta}^1_\gamma h_\gamma>h_\alpha$. 
  \end{proof}
When $\tilde{\delta}^1$ is known deterministically, the expectations in the proposition can be replaced with probabilities, which leads to a newsvendor-network type equation.
\begin{corollary}
If $\tilde{\delta}^1=\delta^1$ deterministically, and $\mathbb{P}(h_\alpha<\delta^1_\gamma h_\gamma,c^2_\alpha>\tilde{\delta}^2_\alpha, c^2_\gamma>\tilde{\delta}_\gamma^2h_\gamma)=1$, then 
\begin{enumerate}
 \item any $x_\gamma^{1*}\in [0, x_\alpha^0]$ satisfying the equation
\begin{equation}
\delta^1_\gamma h_\gamma\mathbb{P}(x_\gamma^0+\delta^1_\gamma x_\gamma^1<\tilde{d}_\gamma) -h_\alpha\mathbb{P}(x_\alpha^0-x_\gamma^1<\tilde{d}_\alpha,x_\gamma^0+\delta^1_\gamma x_\gamma^1<\tilde{d}_\gamma)=c_\gamma^1
\label{Eq:Case2ac}
\end{equation}
is an optimal cross-training level to work on task $\gamma$.

\item if $\delta^1_\gamma h_\gamma\mathbb{P}(x_\gamma^0+\delta^1_\gamma x_\gamma^1<\tilde{d}_\gamma) -h_\alpha\mathbb{P}(x_\alpha^0-x_\gamma^1<\tilde{d}_\alpha,x_\gamma^0+\delta^1_\gamma x_\gamma^1<\tilde{d}_\gamma)\geq c_\gamma^1,$ when $x_\gamma^1\in [0, x_\alpha^0]$, then $x_\gamma^{1*}=x_\alpha^0$ is the optimal cross training level to work on task $\gamma$. 

\item if $\delta^1_\gamma h_\gamma\mathbb{P}(x_\gamma^0+\delta^1_\gamma x_\gamma^1<\tilde{d}_\gamma) -h_\alpha\mathbb{P}(x_\alpha^0-x_\gamma^1<\tilde{d}_\alpha,x_\gamma^0+\delta^1_\gamma x_\gamma^1<\tilde{d}_\gamma)\leq c_\gamma^1,$ when $x_\gamma^1\in [0, x_\alpha^0]$, then $x_\gamma^{1*}=0$ is the optimal cross training level to work on task $\gamma$. 
 
\end{enumerate}

Cross-training level to work on task $\alpha$, $x^{1*}_{\alpha}$ is decided using  Proposition~\ref{Pr:Case1a}.
\end{corollary}
 
Equation~(\ref{Eq:Case2ac}) suggests a trade-off between first stage cross-training and opportunity costs in the second stage. The first term on the left-hand side denotes the savings resulting from training one unit more to work on task $\gamma$. On the other hand, this extra cross-training may cause an opportunity cost for $\alpha$, which is denoted by the second term on the left-hand side.
\vspace{-0.1in}

\subsubsection{Case 2.b: $c^2_\alpha<\tilde{\delta}^2_\alpha h_\alpha$ and $c^2_\gamma<\tilde{\delta}_\gamma^2h_\gamma$.}
\vspace{-0.1in}

Now, we consider the situation where it is profitable to cross-train in the second stage and it is more profitable to use the cross-trained workforce for task $\gamma$ when needed. Hence, the case under consideration is a combination of Cases 1.b and 2.a. On regions $\Omega_1, \Omega_2$ and $\Omega_3$, the cross-trained workforce for $\gamma$ is not needed and the second stage cost function resembles the form in Case 1.b. On regions $\Omega_5^a$ and $\Omega_5^b$, the demand for task $\alpha$ workforce is low, so that there is also no competition between different tasks. Hence, the second stage cost function also takes the form in Case 1.b on these regions. On $\Omega_4^a,\Omega_4^b, \Omega_6^a$ and $\Omega_6^b$, two tasks compete for the cross-trained workforce and second stage cross-training is not possible. So the second stage cost function is the same as in Case 2.a. Now, the first and second order optimality conditions can be compiled from the above results to obtain the following proposition.

\begin{proposition}
If $\mathbb{P}(h_\alpha<\delta^1_\gamma h_\gamma,c^2_\alpha<\tilde{\delta}^2_\alpha, c^2_\gamma<\tilde{\delta}_\gamma^2h_\gamma)=1$, and the training effectivenesses before and after demand is realized are the same, i.e., $\tilde{\delta}^1=\tilde{\delta}^2$ with probability one, then 
\begin{enumerate}
 \item any $x_\gamma^{1*}\in [0, x_\alpha^0]$ satisfying the equation
\begin{equation}
c_\gamma^2 \mathbb{P}(\tilde{d}_\alpha\leq x_\alpha^0-x_\gamma^{1*}, \tilde{d}_\gamma \geq x_\gamma^0+\tilde{\delta}^1_\gamma x_\gamma^{1*})+\mathbb{E}[h_\alpha-\tilde{\delta}_\gamma^1 h_\gamma; x_\alpha^0-x_\gamma^{1*}, x_\gamma^0+\tilde{\delta}_\gamma^1 x_\gamma^{1*}]=c_\gamma^1
\label{Eq:Case2b}
\end{equation}
is an optimal cross-training level to work on task $\gamma$.

\item if $c_\gamma^2 \mathbb{P}(\tilde{d}_\alpha\leq x_\alpha^0-x_\gamma^{1}, \tilde{d}_\gamma \geq x_\gamma^0+\tilde{\delta}^1_\gamma x_\gamma^{1*})+\mathbb{E}[h_\alpha-\tilde{\delta}_\gamma^1 h_\gamma; x_\alpha^0-x_\gamma^1, x_\gamma^0+\tilde{\delta}_\gamma^1 x_\gamma^1]\geq c_\gamma^1,$ when $x_\gamma^1\in [0, x_\alpha^0]$, then $x_\gamma^{1*}=x_\alpha^0$ is the optimal cross training level to work on task $\gamma$. 

\item if $c_\gamma^2 \mathbb{P}(\tilde{d}_\alpha\leq x_\alpha^0-x_\gamma^{1}, \tilde{d}_\gamma \geq x_\gamma^0+\tilde{\delta}^1_\gamma x_\gamma^{1*})+\mathbb{E}[h_\alpha-\tilde{\delta}_\gamma^1 h_\gamma; x_\alpha^0-x_\gamma^1, x_\gamma^0+\tilde{\delta}_\gamma^1 x_\gamma^1]\leq c_\gamma^1,$ when $x_\gamma^1\in [0, x_\alpha^0]$, then $x_\gamma^{1*}=0$ is the optimal cross training level to work on task $\gamma$. 
 
\end{enumerate}

The optimal cross-training level for task $\alpha$, $x_\alpha^{1*}$, is determined using Proposition~\ref{Pr:Case1b}.

\end{proposition}

Remember, under the conditions of Case 1.b and Proposition~\ref{Pr:Case1b}, the opportunity cost is a fixed cost and when there is excess demand for both tasks, the workers are always assigned to their original tasks and excess demand is lost. However, under Case 2.b, when there is excess demand for both tasks, one should exploit the opportunity of saving $\tilde{\delta}_\gamma^1h_\gamma-h_\alpha$ per person by assigning cross-trained workforce to work on task $\gamma$. This opportunity is reflected in the second term on the left-hand side of (\ref{Eq:Case2b}). 
\vspace{-0.2in}

\section[]{Numerical Experiments}\label{Sec:NumRes}
\vspace{-0.1in}

In this section, we analyze how varying different parameters affects cross-training policies numerically. We consider the base-case problem in Table~\ref{Tb:DataBC} and then vary the parameters to see how the decisions change. 
The base-case is designed so that task $\alpha$ is the task that mainly needs cross-trained workers and task $\gamma$ only occasionally exceeds the available capacity.

\begin{table}[!h]
\begin{center}
\begin{tabular}{ccc}\hline
Parameter  &Values for $\alpha$ &Values for $\gamma$\\\hline
$x^0$      &5400  &5400\\
$c^1$      &2800 &2800\\
$c^2$      &4000 &4000\\
$\delta^1$ &0.90 &0.90\\
$\delta^2$ &0.90 &0.90\\
$\tilde{d}$ &Uniform(55000, 65000) &Uniform(20000, 60000)\\\hline
\end{tabular}
\end{center}
\vspace{-0.2in}
\caption{Data for Base-Case Scenario of Two-Task Problem}
\label{Tb:DataBC}
\end{table}

In our first experiment, we wish to see how the difference between productivities of first and second stages affects our cross-training costs. To eliminate the effect of cost difference, we set $c_\alpha^2=c^2_\gamma=2800$ and vary the ratio of the productivity factors, i.e., $\delta^2/\delta^1$, for both tasks $\alpha$ and $\gamma$. The results are shown in Figure~\ref{Fg:TCvsDelta}.
\begin{figure}
\begin{center}
\begin{tabular}{cc}
\subfigure[$h$=3500]{\label{Fg:TCvsDelta3500}
\includegraphics[scale=0.4]{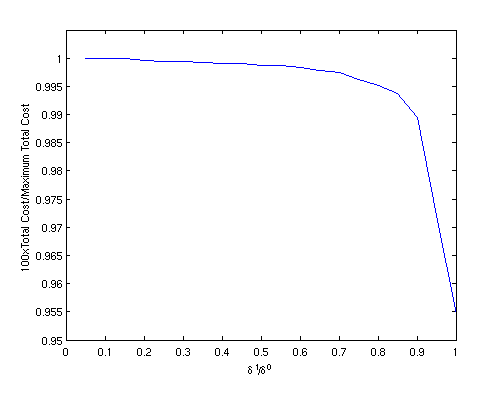}}
& 
\subfigure[$h$=5000]{\label{Fg:TCvsDelta5000}
\includegraphics[scale=0.39]{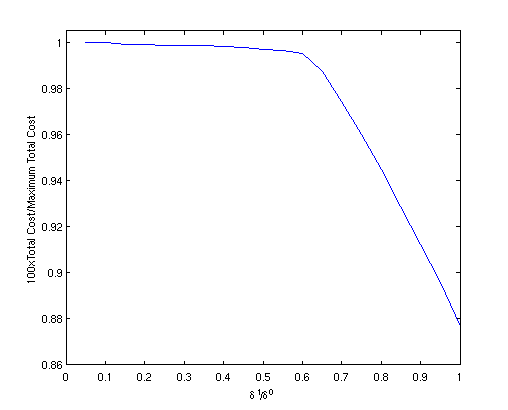}
}
\end{tabular}
\end{center}
\vspace{-0.2in}
\caption{The total cost versus $\delta^2/\delta^1$}
\label{Fg:TCvsDelta}
\end{figure} 
The horizontal axis on Figure~\ref{Fg:TCvsDelta} shows $\delta^2/\delta^1$ and the vertical axis shows the ratio of cost over the cost with no cross-training. We see that for small  $\delta^2$, it is not profitable to cross-train in the second stage. After a critical value of $\delta^2/\delta^1$, the second stage cross-training becomes profitable and we see an almost linear decrease in the cost. The critical value decreases as the opportunity cost increases. 

Now we demonstrate how the variance of demand affects our cross-training policies. Figure~\ref{Fg:FSHCvsDemVarA} shows how the variance of demand for task $\alpha$ affects our decisions when productivity vectors are high, i.e., $\tilde{\delta}^2_\alpha=\tilde{\delta}^2_\gamma=0.9$ with probability one. The demand follows a uniform distribution with the same mean as in Table~\ref{Tb:DataBC} and the horizontal axis denotes the width for the support of demand. The vertical axis shows the cost as a percentage of the optimal cost under deterministic demand. The opportunity costs $h_\alpha$ and $h_\gamma$ are taken to be 8000. The costs exhibit similar behavior when experimented with other values of the opportunity costs.

\begin{figure}
\begin{center}
\begin{tabular}{cc}
\subfigure[The support width is between 0 and 40000]{\label{Fg:TCvsDemvarA8000FSH}
\includegraphics[scale=0.38]{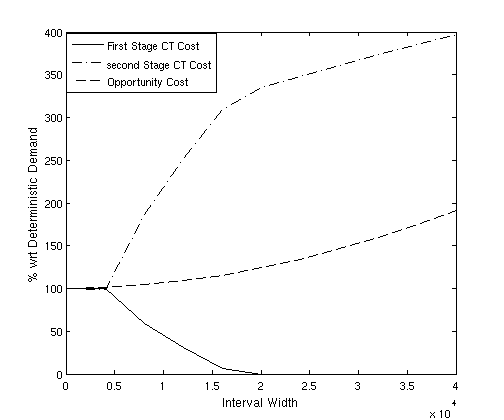}}&
\subfigure[Zoom in when the support width is between 0 and 4000]{\label{Fg:TCvsDemvarA8000FSHZoom}
\includegraphics[scale=0.38]{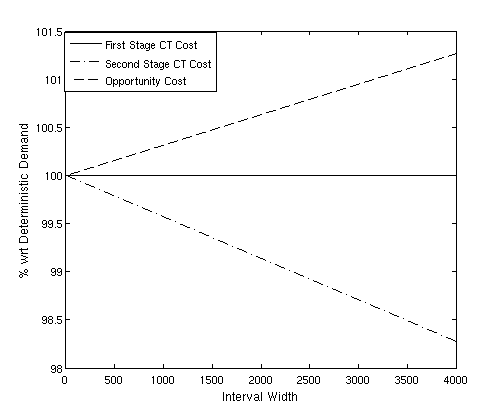}}\\
\end{tabular}
\end{center}
\vspace{-0.2in}
\caption{Individual costs with respect to the variability of demand for $\alpha$ when $\tilde{\delta}^2_{\alpha}=\tilde{\delta}^2_{\gamma}=0.9$}
\label{Fg:FSHCvsDemVarA}
\end{figure}

Figure~\ref{Fg:TCvsDemvarA8000FSH} shows that when the demand for $\alpha$ has a small variance, the individual costs  differ from the deterministic demand only by a small margin. However, after a critical threshold, the first-stage cross-training cost decreases significantly. Finally, if the variance is  too high, then we do not wish to cross-train at all as the condition (b) of Proposition~\ref{Pr:Case1b}. Figure~\ref{Fg:TCvsDemvarA8000FSHZoom} shows how our cross-training policies change before we reach this critical value. For these small variances, we do not change the first stage cross-training as variance increases, but we begin to observe instances where second-stage cross-training is not needed and the number of instances where demand is lost increases as well. Hence, contrary to large variance instances, the second stage cross-training cost decreases and opportunity cost increases with the variance of demand for $\alpha$. Figure~\ref{Fg:TCvsDemVarA} demonstrates that the total cost increases monotonically with the variance of demand for $\alpha$.

\begin{figure}
\begin{center}
\includegraphics[scale=0.38]{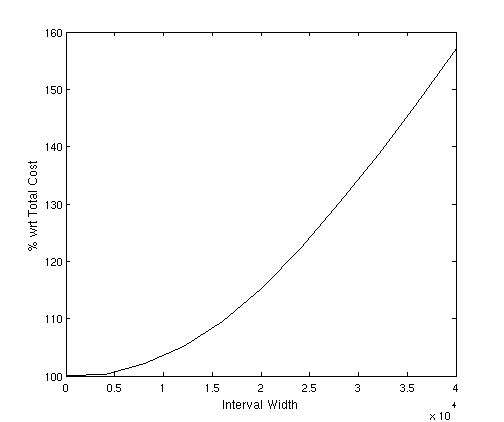}
\end{center}
\vspace{-0.2in}
\caption{Total cost with respect to the variability of demand for $\alpha$ when $\tilde{\delta}^2_{\alpha}=\tilde{\delta}^2_{\gamma}=0.9$}
\label{Fg:TCvsDemVarA}
\end{figure}

Corallary~\ref{Cr:Var} proves that the first stage cross-training  is non-increasing with the demand variance under certain conditions. Hence, we perform experiments to test how restrictive these conditions are. To satisfy the assumption of no second stage cross-training, we take $c^2_\alpha=c^2_\gamma=5000$, $\delta^2_\alpha=0.9$ and $\delta^2_\gamma=0.6$. In Figure~\ref{Fg:FSHCvsDemVar06}, we demonstrate the results of these experiments for $h_i$ ranging from 5000 to 8000. The assumption $c^1_i\geq \tilde{\delta}^1_ih_i$ does not hold for both tasks for these opportunity costs. When $h_\alpha=h_\gamma=5000, 6000$ or 7000, the first stage cross-training cost is non-increasing. However, when $h_\alpha=h_\gamma=8000$, we observe a non-monotonic behavior in the first stage cross-training cost. This behavior can be explained as follows: We see that an increase in the variance forces the decision maker to decrease the first-stage cross-training level as there is more chance that cross-trained workers may not be needed when demand is observed. However, when $\tilde{\delta}^1_i\neq \tilde{\delta}^2_i$ with a positive probability, some capacity is lost by not training in the first-stage. When the opportunity cost and the demand variance is really high, the decision maker cannot effort to lose this capacity. Hence, it is better to take the risk of unnecessary cross-training rather than facing a high opportunity cost. Hence, we conclude that $c^1_i\geq \tilde{\delta}^2_ih_i$ is  sufficient for monotonicity but not necessary.
 
\begin{figure}
\begin{center}
\begin{tabular}{cc}
\subfigure[$h=5000$]{\label{Fg:TCvsDemvarA06_5000FSH}
\includegraphics[scale=0.38]{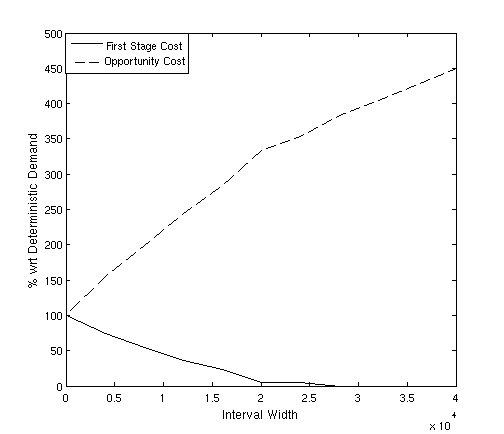}}&\subfigure[$h=6000$]{\label{Fg:TCvsDemvarA06_6000FSH}
\includegraphics[scale=0.38]{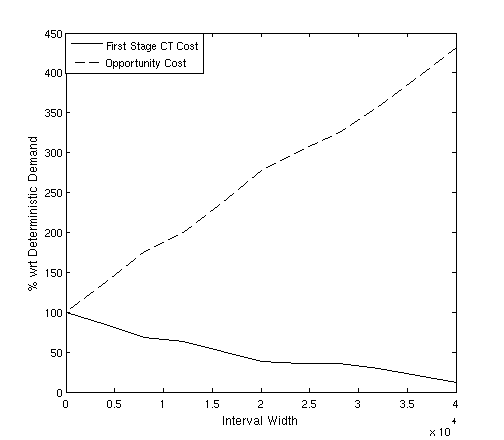}}\\
\subfigure[$h=7000$]{\includegraphics[scale=0.38]{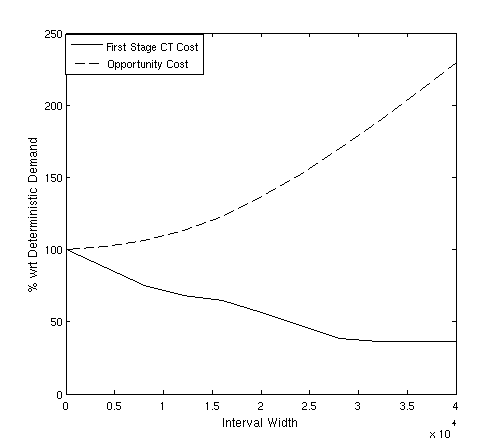}}&\subfigure[$h=8000$]{\label{Fg:TCvsDemvarA06_8000FSH}\includegraphics[scale=0.38]{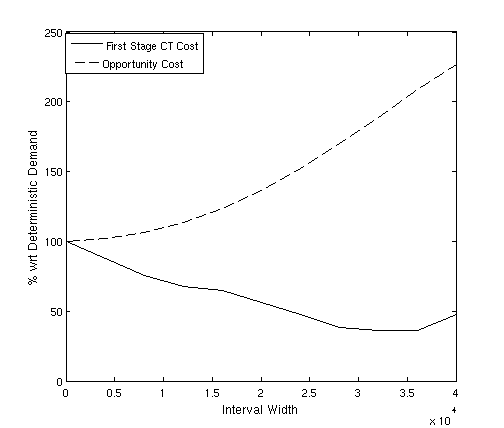}}
\end{tabular}
\end{center}
\vspace{-0.2in}
\caption{Individual costs with respect to the variability of demand for $\alpha$ when $\tilde{\delta}^2_{\alpha}=\tilde{\delta}^2_{\gamma}=0.6$}
\label{Fg:FSHCvsDemVar06}
\end{figure}
\begin{figure}
\begin{center}
\begin{tabular}{cc}
\subfigure[First and second stage cross-training costs]{\label{Fg:TCvsDemvarG09_5000FS}
\includegraphics[scale=0.38]{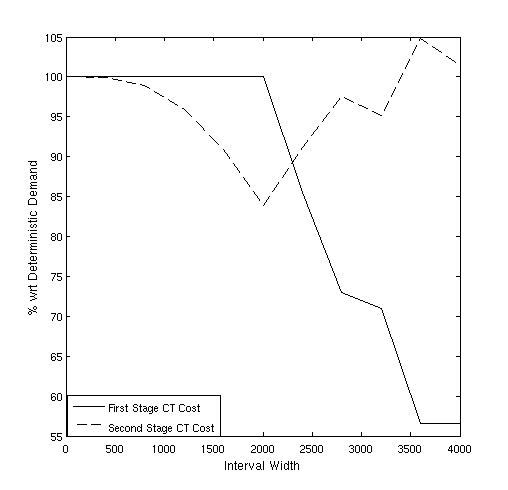}$\quad\quad$}&
\subfigure[Opportunity cost]{\label{Fg:TCvsDemvarG09_5000H}
\includegraphics[scale=0.38]{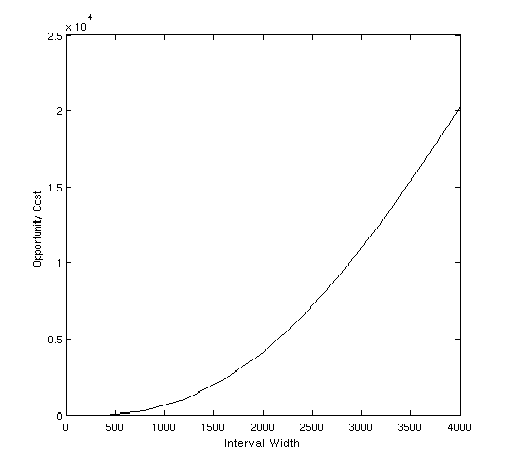}}
\end{tabular}
\end{center}
\vspace{-0.2in}
\caption{Individual costs with respect to the variability of demand for $\gamma$ when $\tilde{\delta}^2_{\alpha}=\tilde{\delta}^2_{\gamma}=0.9$}
\label{Fg:FSHCvsDemVar09}
\end{figure}
Figure~\ref{Fg:FSHCvsDemVar09} shows how the variance of demand for $\gamma$ affects the cross-training policies when the second  stage productivity is as high as the first stage productivity. The variability of $\gamma$ demand has an influence over cross-training policies only when variance is significant. Moderate variances do not cause a change in our first stage cross-training policies. This variability is mitigated by increasing the opportunity cost and the second stage cross-training decreases. When the variance is significant, the first stage cross-training cost starts decreasing and as a result second stage cross-training cost increases together with the opportunity cost. 


Another important factor determining our cross-training policies is the variability of productivity factors. We start by investigating the effect of the variability of the first stage productivity factor.  Similar to our experiments with demand variability, we assume uniform distribution for the first stage productivity factors, fix $\mathbb{E}[\tilde{\delta}^1_\alpha]=\mathbb{E}[\tilde{\delta}^1_\gamma]=0.75$, and vary the width of the support of these random variables. We assume a relatively high opportunity cost of $h_\alpha=h_\gamma=8000$. Figure~\ref{Fg:TCvsDeltaF075_8000FS} demonstrates how the individual costs behave when the second stage productivity factor is deterministically equal to the expected value of the first stage productivity factor. We observe that the first stage cross-training decreases monotonically and this decrease is mitigated by an increase in the second stage cross-training.  

\begin{figure}[h]
\begin{center}
\begin{tabular}{cc}
\subfigure[$\tilde{\delta}^2_\alpha=\tilde{\delta}^2_\gamma=0.75$ ]{\label{Fg:TCvsDeltaF075_8000FS}
\includegraphics[scale=0.38]{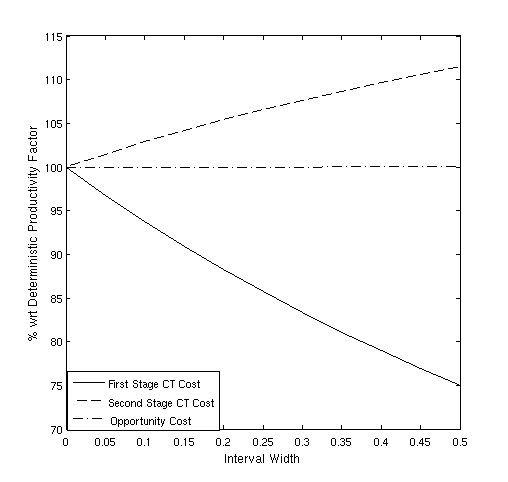}}&
\subfigure[$\tilde{\delta}^2_\alpha=\tilde{\delta}^2_\gamma=0.65$ ]{\label{Fg:TCvsDeltaF065_8000FS}
\includegraphics[scale=0.38]{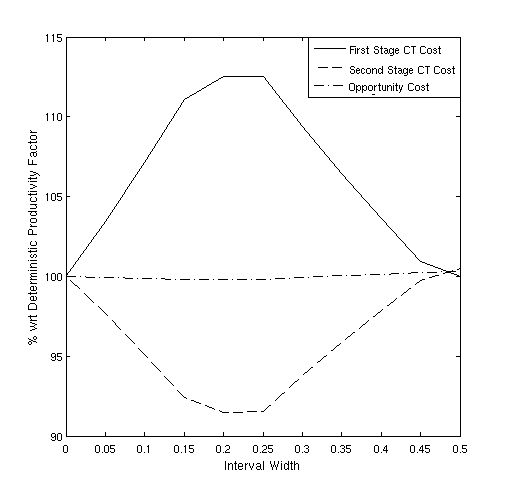}}
\end{tabular}
\end{center}
\vspace{-0.2in}
\caption{Costs with respect to the variability of the first stage productivity factor}
\label{Fg:FSHCvsDelta75}
\end{figure}

\begin{figure}[!h]
\begin{center}
\begin{tabular}{cc}
\subfigure[$h_\alpha=h_\gamma=6000$]{\label{Fg:TCvsDeltaS075_6000FS}
\includegraphics[scale=0.4]{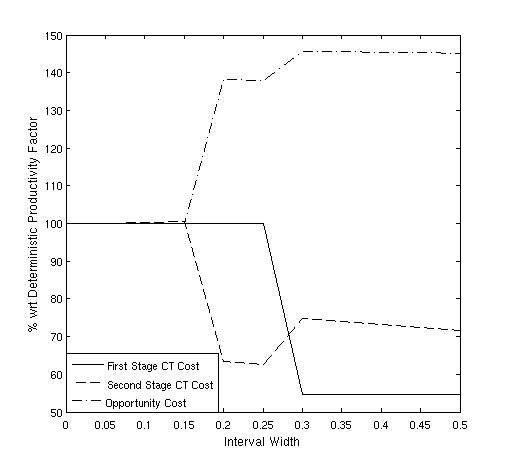}}&
\subfigure[$h_\alpha=h_\gamma=8000$]{\label{Fg:TCvsDeltaS075_8000FS}
\includegraphics[scale=0.4]{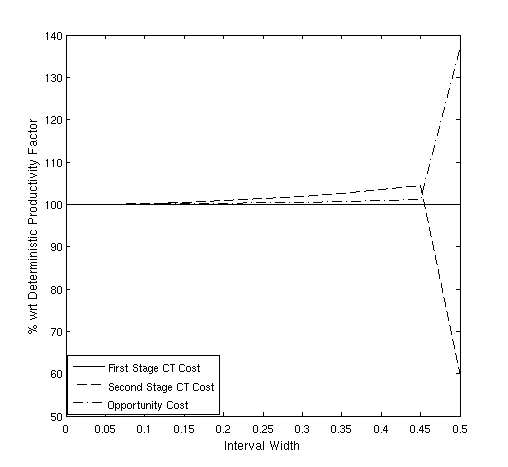}}
\end{tabular}
\end{center}
\vspace{-0.2in}
\caption{Costs with respect to the variability of the second stage productivity factor}
\label{Fg:FSHCvsDeltaS75}
\end{figure}

\begin{figure}[!h]
\begin{center}
\vspace{-0.3in}
\begin{tabular}{cc}
\subfigure[$h_\alpha=h_\gamma=6000$]{\label{Fg:TCvsDeltaS075_6000TC}
\includegraphics[scale=0.4]{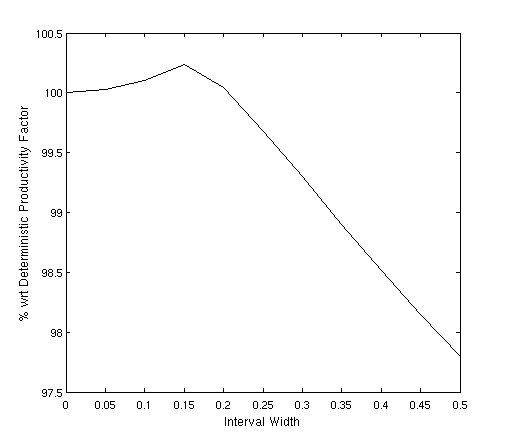}}&
\subfigure[$h_\alpha=h_\gamma=8000$]{\label{Fg:TCvsDeltaS075_8000TC}
\includegraphics[scale=0.39]{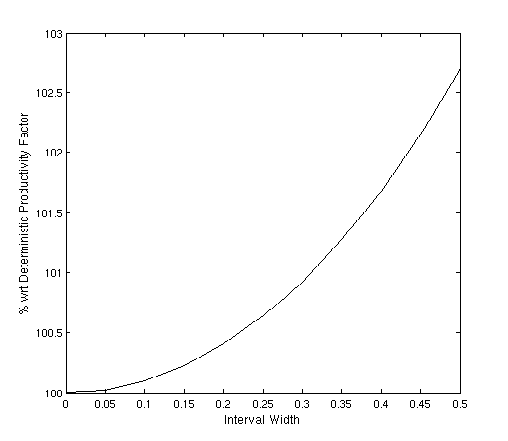}}
\end{tabular}
\end{center}
\vspace{-0.2in}
\caption{Total cost with respect to the variability of the second stage productivity factor}
\label{Fg:TCvsDeltaS75}
\end{figure}

However, when the second stage cross-training productivity factor is lower than the expected value of the first stage, the cross-training policies exhibit a different behavior. Figure~\ref{Fg:TCvsDeltaF065_8000FS} demonstrates our experiments when $\tilde{\delta}^2_\alpha=\tilde{\delta}^2_\gamma=0.65$ with probability one. We observe that the first stage cross-training increases when the variance is small and decreases for higher variances. The intuitive reasoning for this behavior is as follows: When the variance of the first-stage productivity factor is small, cross-training workers in the first stage increases the total available capacity for every realization of the productivity factor and this extra capacity can be used. However, when the variance is high, we can observe really low productivity factors which may make it more beneficial to cross-train in the second stage. When the productivity factors are observed really high, there may not be enough demand to justify such a high capacity. Hence, if the variance increases above a certain threshold, it is beneficial to decrease first-stage cross-training.  


Next, we investigate how the variability of the second stage productivity factor affects the cross-training policies for low ($h_\alpha=h_\gamma=6000$) and high ($h_\alpha=h_\gamma=8000$) opportunity costs. The first stage productivity factor is assumed to be deterministic and the expected value of the second stage is taken to be equal to the first stage productivity. Figure~\ref{Fg:TCvsDeltaS075_6000FS} indicates that the first stage cross-training is not affected for low variability when opportunity cost is low. However, when there is a significant probability that the second stage cross-training can be more effective than the first stage, the optimal first stage cross-training decreases. The second stage cross-training also tends to decrease until we hit this critical probability. When the opportunity cost is high, the first stage cross-training is not affected and the small variability yields an increase in the second stage cross-training as demonstrated in Figure~\ref{Fg:TCvsDeltaS075_8000FS}. However, incurring opportunity cost becomes preferable for higher variability.

Proposition~\ref{Pr:Convdelta} indicates that when the second stage productivity factor for task $i$ ($i=\alpha,\gamma$) lies in $[c^2_i/h_i,1]$ the total cost is a non-decreasing function of the variance of $\tilde{\delta}^2_i$. Figure~\ref{Fg:TCvsDeltaS075_8000TC} demonstrates this behavior when $h_\alpha=h_\gamma=8000$. However, when $h_\alpha=h_\gamma=6000$, there is a positive probability that the second stage productivity can lie outside this interval. Figure~\ref{Fg:TCvsDeltaS075_6000TC} indicates that the total cost can be decreasing when this is the case.

\vspace{-0.2in} 

\section{Concluding Remarks}
\vspace{-0.1in}
In this work, we have studied the effect of imperfect training schemes on the cross-training policies. We have considered a two-stage model, where the workers can be cross-trained offline in the first stage, before the demand is realized, and online in the second stage as the demand is revealed. The cross-trained workers are assumed to be less productive than the workers who are originally trained to do a specific task and the productivity of the cross-trained workers may depend on when they are cross-trained. Surprisingly, when the first stage and second stage training schemes are equally effective, then the cross-training decisions are independent of the hiring or opportunity costs. When the first and second stage training policies differ in their effectiveness, the structure of the problem changes significantly. First, the objective function is not convex anymore and hence the newsvendor-type equations we propose provide necessary conditions for optimality but they are not sufficient. 

We have also analyzed how the variability of demand and productivity factors affect our cross-training decisions. We have shown that under some mild conditions, we tend to cross-train less in the first stage as the demand or productivity factors become more variable. However, when these conditions do not hold, we show via counter-examples that we may wish to increase first stage cross-training as variability increases. Another interesting point is that all the counter-examples we found in this context assume that the training effectiveness is different for first and second stages. 

The insights provided in this paper can be used to devise effective solution methods to address the case when there are more than three tasks. We have also developed a two-stage stochastic integer program to aid decision makers to design cross-training policies in the presence of multiple tasks. We do not present this model in this paper to keep the focus on managerial insights. The integer programming model is available from the authors upon request. 
\vspace{-0.1in}


%
%
%


\bibliographystyle{ormsv080}
\bibliography{Bibliography}

\begin{thebibliography}{25}
\expandafter\ifx\csname natexlab\endcsname\relax\def\natexlab#1{#1}\fi
\expandafter\ifx\csname url\endcsname\relax
  \def\url#1{{\tt #1}}\fi
\expandafter\ifx\csname urlprefix\endcsname\relax\def\urlprefix{URL }\fi
\expandafter\ifx\csname urlstyle\endcsname\relax
  \expandafter\ifx\csname doi\endcsname\relax
  \def\doi#1{doi:\discretionary{}{}{}#1}\fi \else
  \expandafter\ifx\csname doi\endcsname\relax
  \def\doi{doi:\discretionary{}{}{}\begingroup \urlstyle{rm}\Url}\fi \fi

\bibitem[{Bassamboo et~al.(2009)Bassamboo, Randhawa, and {Van
  Mieghem}}]{basranvan09a}
Bassamboo, A., R.S. Randhawa, J.A. {Van Mieghem}. 2009.
\newblock A little flexibility is all you need: On the value of flexible
  resources in queueing systems Working paper.

\bibitem[{Bassamboo et~al.(2010)Bassamboo, Randhawa, and {Van
  Mieghem}}]{basranvan09b}
Bassamboo, A., R.S. Randhawa, J.A. {Van Mieghem}. 2010.
\newblock Optimal flexibility configurations in newsvendor networks: going
  beyond chaining and pairing.
\newblock {\it Management Science\/} {\bf 56} 1285--1303.

\bibitem[{Brusco and Johns(1998)}]{brujoh98}
Brusco, M.J., T.R. Johns. 1998.
\newblock Staffing multiskilled workforce with varying levels of productivity:
  An analysis of cross-training policies.
\newblock {\it Decision Sciences\/} {\bf 29}(2) 499--515.

\bibitem[{Campbell(1999)}]{campbell99}
Campbell, G.M. 1999.
\newblock Cross-utilization of workers whose capabilities differ.
\newblock {\it Management Science\/} {\bf 45}(5) 722--732.

\bibitem[{Campbell(2011)}]{Campbell11}
Campbell, G.M. 2011.
\newblock A two-stage stochastic program for scheduling and allocating
  cross-trained workers.
\newblock {\it Journal of the Operational Research Society\/} {\bf 62}
  1038--1047.

\bibitem[{Chakravarthy and Agnihothri(2005)}]{chaagni}
Chakravarthy, S.R., S.R. Agnihothri. 2005.
\newblock Optimal workforce mix in service systems with two types of customers.
\newblock {\it Production \& Operations Management\/} {\bf 14}(2) 218--231.

\bibitem[{Chou et~al.(2010)Chou, Chua, and Teo}]{cct10}
Chou, M.C., G.A. Chua, C-P. Teo. 2010.
\newblock On range and response: Dimensions of process flexibility.
\newblock {\it European Journal of Operational Research\/} {\bf 207} 711--724.

\bibitem[{Davis et~al.(2009)Davis, Kher, and Wagner}]{DavisKher09}
Davis, D.D., H.V. Kher, B.J. Wagner. 2009.
\newblock Influence of workload imbalances on the need for worker flexibility.
\newblock {\it Computers and Industrial Engineering\/} {\bf 57} 319--329.

\bibitem[{Gel et~al.(2007)Gel, Hopp, and {Van Oyen}}]{gelhopvan07}
Gel, E.S., W.J. Hopp, M.P. {Van Oyen}. 2007.
\newblock Hierarchical cross-training in work-in-process-constrained systems.
\newblock {\it IIE Transactions\/} {\bf 39} 125--143.

\bibitem[{Gnanlet and Wendell(2009)}]{Gnanlet09}
Gnanlet, A., G.G. Wendell. 2009.
\newblock Sequential and simultaneous decision making for optimizing health
  care resource flexibilities.
\newblock {\it Decision Sciences\/} {\bf 40} 295--326.

\bibitem[{Hopp et~al.(2004)Hopp, Tekin, and {Van Oyen}}]{hoptekvan04}
Hopp, W.J., E.~Tekin, M.P. {Van Oyen}. 2004.
\newblock Benefits of skill chaining in serial production lines with
  cross-trained workers.
\newblock {\it Management Science\/} {\bf 50}(1) 83--98.

\bibitem[{Hopp and {Van Oyen}(2004)}]{hoppvanoyen04}
Hopp, W.J., M.P. {Van Oyen}. 2004.
\newblock Agile workforce evaluation: a framework for cross-training and
  coordination.
\newblock {\it IIE Transactions\/} {\bf 36} 919--940.

\bibitem[{Iravani et~al.(2007)Iravani, Kolfal, and {Van Oyen}}]{irakolvan07}
Iravani, S.M.R., B.~Kolfal, M.P. {Van Oyen}. 2007.
\newblock Call-center labor cross-training: it'\unskip s a small world after
  all.
\newblock {\it Management Science\/} {\bf 53}(7) 1102--1112.

\bibitem[{Jordan and Graves(1995)}]{jorgra95}
Jordan, W.C., S.C. Graves. 1995.
\newblock Principles on the benefits of manufacturing process flexibility.
\newblock {\it Management Science\/} {\bf 41}(4) 577--594.

\bibitem[{Jordan et~al.(2004)Jordan, Inman, and Blumenfeld}]{jorinblu04}
Jordan, W.C., R.R. Inman, D.E. Blumenfeld. 2004.
\newblock Chained cross-training of workers for robust performance.
\newblock {\it IIE Transactions\/} {\bf 36} 953--967.

\bibitem[{Lyons(1992)}]{lyons92}
Lyons, R.F. 1992.
\newblock Cross-training: A richer staff for leaner budgets.
\newblock {\it Nursing Management\/} {\bf 23}(1) 43--44.

\bibitem[{McCune(1994)}]{mccune94}
McCune, J.C. 1994.
\newblock On the train gang.
\newblock {\it Management Review\/} {\bf 83}(10) 57--60.

\bibitem[{Netessine et~al.(2002)Netessine, Dobson, and Shumsky}]{netdobshu02}
Netessine, S., G.~Dobson, R.A. Shumsky. 2002.
\newblock Flexibly service capacity: Optimal investment and the impact of
  demand correlation.
\newblock {\it Operations Research\/} {\bf 50}(2) 375--388.

\bibitem[{Pinker and Shumsky(2000)}]{pinshu00}
Pinker, E.J., R.A. Shumsky. 2000.
\newblock The efficiency-quality trade-off of cross-trained workers.
\newblock {\it Manufacturing \& Service Operations Management\/} {\bf 2}(1)
  32--48.

\bibitem[{Say{\i}n and Karabat{\i}(2007)}]{saykar07}
Say{\i}n, S., S.~Karabat{\i}. 2007.
\newblock Assigning cross-trained workers to departments: A two-stage
  optimization model to maximize utility and skill improvement.
\newblock {\it European Journal of Operational Research\/} {\bf 176}
  1643--1658.

\bibitem[{Tanrisever et~al.(2012)Tanrisever, Morrice, and
  Morton}]{Tanrisever2012334}
Tanrisever, F., D.~Morrice, D.~Morton. 2012.
\newblock Managing capacity flexibility in make-to-order production
  environments.
\newblock {\it European Journal of Operational Research\/} {\bf 216}(2) 334 --
  345.

\bibitem[{Vairaktarakis and Winch(1999)}]{vaiwin99}
Vairaktarakis, G., J.K. Winch. 1999.
\newblock Worker cross-training in paced assembly lines.
\newblock {\it Manufacturing \& Service Operations Management\/} {\bf 1}(2)
  112--131.

\bibitem[{{Van Mieghem}(1998)}]{van98}
{Van Mieghem}, J.A. 1998.
\newblock Investment strategies for flexible resources.
\newblock {\it Management Science\/} {\bf 44}(8) 1071--1078.

\bibitem[{{Van Mieghem} and Rudi(2002)}]{vanrud02}
{Van Mieghem}, J.A., N.~Rudi. 2002.
\newblock Newsvendor networks: dynamic inventory management and capacity
  investments with discretionary pooling.
\newblock {\it Manufacturing \& Service Operations Management\/} {\bf 4}(4)
  313--335.

\bibitem[{Walsh et~al.(2000)Walsh, Carlyle, Fowler, and Wee}]{wcfw00}
Walsh, M., W.M. Carlyle, J.W. Fowler, C.~Wee. 2000.
\newblock Modeling operator cross-training in wafer fabrication given uncertain
  demand: A stochastic programming approach.
\newblock J.~K. Cochran, J.~W. Fowler, S.~Brown, eds., {\it Proc. of the Int.
  Conf. on Modeling and Analysis of Semiconductor Manufacturing\/}. 226--232.

\end{thebibliography}

\end{document}